\documentclass[a4paper,12pt]{amsart}
\usepackage{amsmath}
\usepackage{amsthm}
\usepackage{amsfonts}
\usepackage{amssymb,array}
\usepackage{enumerate}
\usepackage[dvips]{graphicx}
\usepackage{comment}
\usepackage{ccaption}
\usepackage[all]{xy}

\theoremstyle{definition}
\newtheorem{define}{Definition}[section]
\theoremstyle{plain}
\newtheorem{problem}{Problem}
\theoremstyle{plain}
\newtheorem{theorem}{Theorem}[section]
\theoremstyle{plain}
\newtheorem{lemma}[theorem]{Lemma}
\theoremstyle{plain}
\newtheorem{proposition}[theorem]{Proposition}
\theoremstyle{plain}
\newtheorem{corollary}[theorem]{Corollary}
\theoremstyle{remark}
\newtheorem{remark}[theorem]{Remark}
\theoremstyle{plain}
\newtheorem{fact}[define]{Fact}
\theoremstyle{definition}
\newtheorem{step}{Step}

\theoremstyle{definition}
\newtheorem{case}{Case}

\numberwithin{equation}{section}
\numberwithin{figure}{section}
\numberwithin{table}{section}

\renewcommand{\theenumi}{\arabic{enumi}}

\title[Visible actions and multiplicity-free criteria]
{Visible actions 
and criteria for multiplicity-freeness of representations of Heisenberg groups}
\author[A. Baklouti, A. Sasaki]{Ali Baklouti, Atsumu Sasaki}

\thanks{
	The first author was supported by the DGRST, Research Laboratory LR 11E S52, 
	and the second author was partially supported by 
	JSPS KAKENHI Grant Number JP19K03453. 
}


\address[A. Baklouti]{Department of Mathematics, Faculty of Science of Sfax, 
Route de Soukra, 3038, Sfax, Tunisia. }
\email{ali.baklouti@fss.usf.tn}

\address[A. Sasaki]{Department of Mathematics, Faculty of Science, Tokai University, 
4-1-1, Kitakaname, Hiratsuka, Kanagawa, 259-1292, Japan 
and Department of Mathematics, FAU Erlangen-N\"urnberg, 
Cauerstrasse 11, 91058, Erlangen, Germany}
\email{atsumu@tokai-u.jp}

\date{\today}
\begin{document}

\begin{abstract}
	A visible action on a complex manifold is a holomorphic action 
	that admits a $J$-transversal totally real submanifold $S$. 
	It is said to be strongly visible if there exists an orbit-preserving anti-holomorphic 
	diffeomorphism $\sigma $ such that $\sigma |_S = \operatorname{id}_S$. 
	Let $G$ be the Heisenberg group and $H$ a non-trivial connected closed subgroup of $G$. 
	We prove that any complex homogeneous space
	$D = G^{\mathbb{C}}/H^{\mathbb{C}}$ admits a strongly visible $L$-action, 
	where $L$ stands for a connected closed subgroup of $G$ explicitly 
	constructed through a co-exponential basis of $H$ in $G$. 
	This leads in turn that $G$ itself acts strongly visibly on $D$.
	The proof is carried out by finding explicitly an orbit-preserving anti-holomorphic diffeomorphism 
	and a totally real submanifold $S$, 
	for which the dimension depends upon  the dimensions of $G$ and $H$. 
	As a direct application, 
	our geometric results provide a proof of various multiplicity-free theorems 
	on continuous representations on the space of holomorphic sections on $D$. 
	Moreover, we also generate as a consequence, 
	a geometric criterion for a quasi-regular representation of $G$ to be multiplicity-free. 
\end{abstract}

\maketitle

\tableofcontents


\section{Introduction}
\label{sec:intro}

This paper investigates a classification problem on (strongly) visible actions 
on complex homogeneous spaces of Heisenberg groups. 

The notion of (strongly) visible actions on complex manifolds 
has been introduced by T. Kobayashi \cite{ko04,ko05}. 
A holomorphic action of a Lie group $L$ on a connected complex manifold $D$ 
is called {\it strongly visible} (\cite[Definition 3.3.1]{ko05}) 
if there exist a real submanifold $S$ in $D$ and an anti-holomorphic diffeomorphism $\sigma$ on $D$ 
such that the conditions (\ref{v0})--(\ref{s2}) are satisfied: 
\begin{gather}
\tag{V.0}
\label{v0}
\mbox{$D':=L\cdot S$ is a non-empty open set in $D$}, \\
\tag{S.1}
\label{s1}
\sigma |_S=\operatorname{id}_S,\\
\tag{S.2}
\label{s2}
\mbox{$\sigma$ preserves each $L$-orbit in $D'$. }
\end{gather}

A real submanifold $S$ satisfying (\ref{v0}) and (\ref{s1}) is called a {\it slice} 
for the strongly visible $L$-action on $D$. 
A slice $S$ satisfying (\ref{s1}) is totally real, 
namely, $T_xS\cap J_x(T_xS)=\{ 0\} $ for any $x\in S$ 
where $J$ is a complex structure on $D$ (see \cite[Remark 3.3.2]{ko05}). 
Further, a strongly visible action is visible \cite[Definition 3.1.1]{ko05} 
in the sense that (\ref{v0}) with $J_x(T_xS)\subset T_x(L\cdot x)$ for any $x\in S$ 
(see \cite[Theorem 4]{ko05}). 
We allow that a slice $S$ meets every $L$-orbit in $D'$ twice or more.
Namely, $S$ is not necessary a cross-section of $L$-orbits in $D'$. 

Recently, 
strongly visible actions has been studied in various settings, 
such as Hermitian symmetric spaces \cite{ko07}, 
flag varieties \cite{ko07-japan,tanaka12}, 
complex linear spaces \cite{sa09}, 
nilpotent orbits in complex simple Lie algebras \cite{sa16}, 
and reductive non-symmetric spherical homogeneous spaces (for example, \cite{sa10a}), 
which are in connection with multiplicity-free representations of reductive Lie groups. 
Under these circumstances, 
a kind of decomposition theorems of reductive Lie groups and reductive homogeneous spaces, 
Cartan decomposition and Iwasawa decomposition for instance, 
play a crucial role to explicitly build the corresponding slices. 
However, strongly visible actions of nilpotent and solvable Lie groups are not well-known. 
Indeed, in contrast to the case of reductive Lie groups, 
there is no analogue of decomposition theorems for nilpotent and solvable Lie groups. 

Let us bring up our problems as follows. 
Let $G$ be a connected and simply connected nilpotent (resp. solvable) Lie group and 
$H$ a non-trivial connected closed subgroup of $G$. 
We denote by $G^{\mathbb{C}}$ and $H^{\mathbb{C}}$ the complexifications of $G$ and $H$, 
respectively. 
We say that $G^{\mathbb{C}}/H^{\mathbb{C}}$ is a complex nilpotent (resp. solvable) homogeneous space. 
Then, $G$ naturally acts on the complex manifold $G^{\mathbb{C}}/H^{\mathbb{C}}$ holomorphically.  We consider the following problems:

\begin{problem}
\label{p:visible}
Find pairs $(G,H)$ 
such that the $G$-action on $G^{\mathbb{C}}/H^{\mathbb{C}}$ is strongly visible. 
\end{problem}

For a strongly visible $G$-action on $G^{\mathbb{C}}/H^{\mathbb{C}}$, 
we can take a slice $S$ and an anti-holomorphic diffeomorphism $\sigma$ satisfying (\ref{v0})--(\ref{s2}). 
Then, the dimension of $S$ is not greater than $\dim G/H$ because $S$ is a totally real submanifold 
in $G^{\mathbb{C}}/H^{\mathbb{C}}$, 
whereas, 
this is at least the codimension 
of generic $G$-orbits in $G^{\mathbb{C}}/H^{\mathbb{C}}$ (cf. \cite[Lemma 3.2.1]{ko05}). 
In this sense, 
we say that a slice $S$ is smallest if $\dim S$ coincides with the codimension of generic orbits. 
Then, we have: 

\begin{problem}
\label{p:slice}
Construct explicitly a small slice $S$ for a strongly visible $G$-action on $G^{\mathbb{C}}/H^{\mathbb{C}}$. 
\end{problem}

The present paper solves Problems \ref{p:visible} and \ref{p:slice} in the setting of the Heisenberg group. 
First, we prove: 

\begin{theorem}
\label{thm:main}
Let $G$ be the Heisenberg group 
and $H$ a non-trivial connected closed subgroup of $G$. 
Then, there exists a connected closed subgroup $L$ of $G$ such that 
the $L$-action on $G^{\mathbb{C}}/H^{\mathbb{C}}$ is strongly visible. 
\end{theorem}

Indeed, we specify `smallest' $L$ 
and give concrete descriptions of a `smallest' slice $S$ 
and an anti-holomorphic diffeomorphism $\sigma $, respectively, 
for the $L$-action on $G^{\mathbb{C}}/H^{\mathbb{C}}$ in the proof of Theorem \ref{thm:main} 
(see also Table \ref{table:choice} for our choice of $L$, $S$ and $\sigma$ for $H$). 

The strongly visible action of a small group ensures that of a large one, 
namely, we obtain: 

\begin{corollary}
\label{cor:main}
For any non-trivial connected closed subgroup $H$, 
the $G$-action on $G^{\mathbb{C}}/H^{\mathbb{C}}$ is strongly visible. 
\end{corollary}

Hence, Corollary \ref{cor:main} gives an answer to Problem \ref{p:visible} 
in the case of the Heisenberg group. 

We also focus attention on the relationship 
between strongly visible actions on nilpotent (resp. solvable) homogeneous spaces 
and the multiplicity-freeness of some related representations. 
Originally, the notion of (strongly) visible actions has been introduced as a geometric condition 
of the propagation theory of multiplicity-freeness property \cite{ko13} 
(see also \cite{ko04,ko05} and Fact \ref{fact:propagation}). 
This allows to generate a unified explanation of multiplicity-freeness property for representations 
which have been found independently. 
In fact, 
this new line of investigation was initiated in \cite{ko08} 
for unitary highest weight representations of scalar type. 
Moreover, if we find a strongly visible action on a complex manifold, 
we expect to get various multiplicity-free representations. 

Indeed, our main results given in Theorem \ref{thm:main} and Corollary \ref{cor:main} yield 
a geometric explanation of multiplicity-freeness property of the continuous representation 
on the space of holomorphic functions due to the propagation theory, 
which we will explain in Theorem \ref{thm:multiplicity-free}. 

We also seek another geometric criterion for the quasi-regular representation $\pi _H$ 
(see (\ref{eq:quasi-regular}) for definition) to be multiplicity-free. 
To state our geometric criterion, we give a setup as follows: 
Let $\mathfrak{g},\mathfrak{h}$ be the Lie algebras of the Heisenberg group $G$ 
and a proper connected closed subgroup $H$, respectively. 
We put $d:=\dim G-\dim H$ (then $0<d<\dim G$) 
and take a co-exponential basis $\{ X_1,\ldots ,X_d\} $ to $\mathfrak{h}$ in $\mathfrak{g}$. 
We set a complementary subspace $\mathfrak{q}$ of $\mathfrak{h}$ in $\mathfrak{g}$ 
as $\mathfrak{q}=\operatorname{span}_{\mathbb{R}}\{ X_1,\ldots ,X_d\} $ and 
$Q:=\langle \exp \mathfrak{q}\rangle$ as a closed subgroup of $G$ generated by $\exp \mathfrak{q}$. 
Then, we prove: 

\begin{theorem}[see Theorem \ref{thm:criterion} for detail]
\label{thm:main-rep}
For a non-trivial connected closed subgroup $H$ of the Heisenberg group $G$, 
the following conditions are equivalent: 
\begin{enumerate}
	\renewcommand{\theenumi}{\roman{enumi}}
	\item The quasi-regular representation $\pi _H$ of $G$ is multiplicity-free. 
	\item The $Q$-action on $G^{\mathbb{C}}/H^{\mathbb{C}}$ is strongly visible. 
\end{enumerate}
\end{theorem}

In order to verify the multiplicity-freeness of $\pi _H$, 
we shall apply Corwin--Greenleaf formula \cite{corwin-greenleaf}  
and calculate the multiplicity of each associated isotypic component of $G$ 
occurring in the irreducible decomposition of $\pi _H$ (see Fact \ref{fact:corwin-greenleaf}). 

Concerning the relationship between visible actions, the Corwin--Greenleaf multiplicity function 
and the Kirillov orbit method, 
we refer to \cite{kobayashi-nasrin} in a case 
where $(G,H)$ is a semisimple symmetric pair of holomorphic type. 

In view of our proof of Theorem \ref{thm:main-rep}, 
we find out a deep relationship between the strongly visible $Q$-action on $G^{\mathbb{C}}/H^{\mathbb{C}}$ 
and the fact that $\pi _H$ is multiplicity-free. 

\begin{theorem}[cf. Theorem \ref{thm:support}]
\label{thm:main-support}
One can construct a slice $S$ for the strongly visible $Q$-action on $G^{\mathbb{C}}/H^{\mathbb{C}}$ 
of dimension equals the rank of the support for multiplicity-free $\pi _H$ 
(see (\ref{eq:rank}) for definition). 
\end{theorem}

It is noteworthy to mention here that 
Theorem \ref{thm:main-support} gives an evidence of \cite[Conjecture 3.2]{ko06} 
affirmatively in the complex Heisenberg homogeneous spaces 
as same as linear actions and nilpotent orbits in complex simple Lie algebras (see \cite{sa09,sa18}). 

This paper is organized as follows. 
In Section \ref{sec:preliminaries}, 
we fix a general setup for the study of strongly visible actions on complex Heisenberg homogeneous spaces 
and prepare two anti-holomorphic diffeomorphisms on them. 
We next explain our strategy of the proof for Theorem \ref{thm:main}, 
which is based on Lemma \ref{lem:reduction}. 
Accordingly, 
we give a proof of Theorem \ref{thm:main} for each closed subgroup $H$ 
in Sections \ref{sec:(0,0,1)}--\ref{sec:(m,0,0)}. 
In particular, we provide a closed subgroup $L$ and a slice $S$ for the $L$-action 
on $G^{\mathbb{C}}/H^{\mathbb{C}}$ satisfying (\ref{v0})--(\ref{s2}) explicitly. 
In Section \ref{sec:proof}, 
we accomplish the proofs of Theorem \ref{thm:main} and Corollary \ref{cor:main}. 
Section \ref{sec:representation} is devoted to present new results on multiplicity-free representations 
as applications of our main theorems.
In Section \ref{subsec:propagation}, 
we record a brief summary on the propagation theory of multiplicity-freeness property \cite{ko13} 
and find multiplicity-free representations as results of Theorem \ref{thm:main} and Corollary \ref{cor:main} 
via this theory. 
In Section \ref{subsec:criterion}, 
we present Theorem \ref{thm:criterion} which covers the results of Theorem \ref{thm:main-rep}. 
In Sections \ref{subsec:mf-h}--\ref{subsec:visible-h}, 
we give a proof of Theorem \ref{thm:criterion}. 
In Section \ref{subsec:support}, 
we establish a result concerning the construction of specific slices 
for strongly visible actions on Heiseberg homogeneous spaces (cf. Theorem \ref{thm:support}). 

The authors would like to thank an anonymous referee for careful comments and suggestions. 


\section{Preliminaries}
\label{sec:preliminaries}

\subsection{Heisenberg Lie algebra and its subalgebras}
\label{subsec:subalgebra}

Let $\mathfrak{g}$ be the $(2n+1)$-dimensional Heisenberg Lie algebra. 
One can take a basis 
\begin{align}
\label{eq:basis}
\mathcal{B}\equiv \mathcal{B}_{(n,n,1)}:=\{ X_1,\ldots ,X_n,Y_1,\ldots ,Y_n,Z\} 
\end{align}
of $\mathfrak{g}$ such that the following relations hold for $1\leq i,j\leq n$: 
\begin{align}
\label{eq:bracket}
[X_i,Y_j]=\delta _{ij}Z,\quad 
[X_i,X_j]=[Y_i,Y_j]=[X_i,Z]=[Y_j,Z]=0. 
\end{align}
Here, $\delta _{ij}$ equals one if $i=j$ and otherwise zero. 
Then, the center $\mathfrak{z}(\mathfrak{g})$ of $\mathfrak{g}$ coincides with $\mathbb{R}Z$. 
Further, the bracket relations (\ref{eq:bracket}) show that 
$[\mathfrak{g},[\mathfrak{g},\mathfrak{g}]]=\{ 0\}$, 
from which $\mathfrak{g}$ is a two-step nilpotent Lie algebra. 

We define subsets $\mathcal{B}_{(p,q,\varepsilon )}$ and $\mathcal{B}_{(m,0,\varepsilon )}$ 
of the basis $\mathcal{B}$ for $1\leq p,q\leq n$, $\varepsilon =0,1$ and $1\leq m\leq n$ by 
\begin{align*}
\mathcal{B}_{(p,q,\varepsilon )}&:=\{ X_1,\ldots ,X_p,Y_1,\ldots ,Y_q,\varepsilon Z\} ,\\
\mathcal{B}_{(m,0,\varepsilon )}&:=\{ X_1,\ldots ,X_m,\varepsilon Z\}  
\end{align*}
and $\mathcal{B}_{(0,0,1)}:=\{ Z\} $. 
For a subset $\mathcal{B}_{(k,\ell ,\varepsilon )}\subset \mathcal{B}$, 
we denote by $\operatorname{span}_{\mathbb{R}}\mathcal{B}_{(k,\ell ,\varepsilon )}$ 
the subspace of $\mathfrak{g}$ with basis $\mathcal{B}_{(k,\ell ,\varepsilon )}$ and set 
\begin{align}
\label{eq:hpqz}
\mathfrak{h}_{(k,\ell ,\varepsilon )}&:=\operatorname{span}_{\mathbb{R}}\mathcal{B}_{(k,\ell ,\varepsilon )}. 
\end{align}
By definition, we also write $\mathfrak{h}_{(m,0,1)}=\mathfrak{h}_{(m,0,0)}+\mathfrak{z}(\mathfrak{g})$. 

The Lie algebra 
$\mathfrak{h}_{(p,q,1)}$ is isomorphic to the direct sum of the Heisenberg Lie algebra 
of dimension $2p+1$ and an abelian subalgebra of dimension $q-p$ if $p<q$, 
and $\mathfrak{h}_{(p,p,1)}$ is itself the Heisenberg Lie algebra of dimension $2p+1$. 
On the other hand, $\mathfrak{h}_{(m,0,1)}$ and $\mathfrak{h}_{(m,0,0)}$ are abelian subalgebras 
for any $1\leq m\leq n$. 

We know from \cite[Proposition 3.1]{bky} that 
a subalgebra of the Heisenberg Lie algebra is characterized as follows: 

\begin{lemma}
\label{lem:subalgebra}
A non-trivial subalgebra of the Heisenberg Lie algebra $\mathfrak{g}$ is isomorphic to 
one of the following subalgebras: 
\begin{enumerate}
	\renewcommand{\theenumi}{\roman{enumi}}
	\item $\mathfrak{h}_{(0,0,1)}=\mathfrak{z}(\mathfrak{g})=\mathbb{R}Z$, 
	\label{alg:(0,0,1)}
	\item $\mathfrak{h}_{(p,q,1)}$ 
	for some $1\leq p\leq q\leq n$, $(p,q)\neq (n,n)$, 
	\label{alg:(p,q,1)}
	\item $\mathfrak{h}_{(m,0,1)}$ for some $1\leq m\leq n$, 
	\label{alg:(m,0,1)}
	\item $\mathfrak{h}_{(m,0,0)}$ for some $1\leq m\leq n$. 
	\label{alg:(m,0,0)}
\end{enumerate}
\end{lemma}

Next, 
let us take a non-trivial subalgebra $\mathfrak{h}$ of the Heisenberg Lie algebra $\mathfrak{g}$ 
with basis $\mathcal{B}_{\mathfrak{h}}$ as a subset of $\mathcal{B}$. 
We set $(\mathcal{B}_{\mathfrak{h}})^c:=\mathcal{B}-\mathcal{B}_{\mathfrak{h}}$ and 
$\mathfrak{q}\equiv C(\mathfrak{h}):=\operatorname{span}_{\mathbb{R}}(\mathcal{B}_{\mathfrak{h}})^c$. 
Then, $\mathfrak{q}$ is a complementary subspace of $\mathfrak{h}$ in $\mathfrak{g}$, 
namely, $\mathfrak{h}\cap \mathfrak{q}=\{ 0\}$ and $\mathfrak{g}=\mathfrak{h}+\mathfrak{q}$. 
In particular, we set 
\begin{align}
\label{eq:q}
\mathfrak{q}_{(k,\ell ,\varepsilon )}
=\operatorname{span}_{\mathbb{R}}(\mathcal{B}_{(k,\ell ,\varepsilon )})^c. 
\end{align}

For each subalgebra given in Lemma \ref{lem:subalgebra}, 
we have 

\begin{lemma}
\label{lem:q}
\begin{enumerate}
	\renewcommand{\theenumi}{\roman{enumi}}
	\item $(\mathcal{B}_{(0,0,1)})^c=\{ X_1,\ldots ,X_n,Y_1,\ldots ,Y_n\}$. 
	\item If $1\leq p\leq q<n$,
	then $(\mathcal{B}_{(p,q,1)})^c=\{ X_{p+1},\ldots ,X_n,Y_{q+1},\ldots ,Y_n\} $, 
	and if $1\leq p<q=n$, then $(\mathcal{B}_{(p,n,1)})^c=\{ X_{p+1},\ldots ,X_n\} $. 
	\item If $1\leq m<n$, 
	then $(\mathcal{B}_{(m,0,1)})^c=\{ X_{m+1},\ldots ,X_n,Y_1,\ldots ,Y_n\} $, 
	and If $m=n$, 
	then $(\mathcal{B}_{(n,0,1)})^c=\{ Y_1,\ldots ,Y_n\} $. 
\end{enumerate}
\end{lemma}

\subsection{The Heisenberg group and its closed subgroups}
\label{subsec:subgroup}

Throughout this paper, 
let $G={\bf{H}}_n$ denote the connected and simply connected Heisenberg group of dimension $2n+1$ 
with Heisenberg Lie algebra $\mathfrak{g}$. 
The exponential map $\exp :\mathfrak{g}\to G$ is a diffeomorphism, 
from which we obtain $G=\exp \mathfrak{g}$. 
In particular, $G$ has a structure which admits a coordinates system of exponential type. 

For convenience, 
we shall use the notation as follows. 
Let $\mathbb{K}=\mathbb{R}$ or $\mathbb{C}$ and $N$ be a positive integer. 
We write $\boldsymbol{x}:=(x_1,\ldots ,x_N),~\boldsymbol{y}:=(y_1,\ldots ,y_N)\in \mathbb{K}^N$. 
Now, we define a symmetric bilinear form on $\mathbb{K}^N$ by 
\begin{align}
\mathbb{K}^N\times \mathbb{K}^N\to \mathbb{K},\quad 
(\boldsymbol{x},\boldsymbol{y})\mapsto (\boldsymbol{x}|\boldsymbol{y}):=x_1y_1+\cdots +x_Ny_N. 
\end{align}

We set 
\begin{align}
\label{eq:g}
g(\boldsymbol{x},\boldsymbol{y},z)
:=e^{x_1X_1}\cdots e^{x_nX_n}e^{y_1Y_1}\cdots e^{y_nY_n}e^ {zZ}. 
\end{align}
Then, we have 
\begin{align*}
G=\{ g(\boldsymbol{x},\boldsymbol{y},z):\boldsymbol{x},\boldsymbol{y}\in \mathbb{R}^n,z\in \mathbb{R}\} . 
\end{align*}

We observe the multiple of $e^{x_iX_i}$ and $e^{y_jY_j}$ in order to express that of two elements of $G$. 

\begin{lemma}
\label{lem:exp}
$e^{y_jY_j}e^{x_iX_i}=e^{x_iX_i}e^{y_jY_j}e^{-\delta _{ij}x_iy_jZ}$ for $1\leq i,j\leq n$. 
\end{lemma}

\begin{proof}
First, let us compute 
\begin{align}
\label{eq:ad}
e^{-x_iX_i}e^{y_jY_j}e^{x_iX_i}=e^{y_j\operatorname{Ad}(e^{-x_iX_i})Y_j}. 
\end{align}
As $[X_i,Y_j]=\delta _{ij}Z$ and $[X_i,[X_i,Y_j]]=0$, 
$\operatorname{Ad}(e^{-x_iX_i})Y_j$ is given by 
\begin{align*}
\operatorname{Ad}(e^{-x_iX_i})Y_j
&=e^{\operatorname{ad}(-x_iX_i)}Y_j\\
&=\sum _{k=0}^{\infty }\frac{(-x_i)^k}{k!}(\operatorname{ad}X_i)^kY_j\\
&=Y_j-\delta _{ij}x_iZ. 
\end{align*}
Thus, the right-hand side of (\ref{eq:ad}) equals $e^{y_j(Y_j-\delta _{ij}x_iZ)}$. 
Moreover, this is $e^{y_jY_j}e^{-\delta _{ij}x_iy_jZ}$ because $[Y_i,Z]=0$. 
Hence, we obtain 
\begin{align*}
e^{y_jY_j}e^{x_iX_i}
&=e^{x_iX_i}(e^{-x_iX_i}e^{y_jY_j}e^{x_iX_i})
=e^{x_iX_i}e^{y_jY_j}e^{-\delta _{ij}x_iy_jZ}. 
\end{align*}
Hence, Lemma \ref{lem:exp} has been proved. 
\end{proof}

\begin{lemma}
\label{lem:multiple}
For $\boldsymbol{x},\boldsymbol{y},\boldsymbol{s},\boldsymbol{t}\in \mathbb{K}^n$ 
and $z,u\in \mathbb{K}$, one has 
\begin{align*}
g(\boldsymbol{x},\boldsymbol{y},z)\cdot g(\boldsymbol{s},\boldsymbol{t},u)
=g(\boldsymbol{x}+\boldsymbol{s},\boldsymbol{y}+\boldsymbol{t},z+u-(\boldsymbol{y}|\boldsymbol{s})), 
\end{align*}
and $g(\boldsymbol{x},\boldsymbol{y},z)^{-1}
=g(-\boldsymbol{x},-\boldsymbol{y},-z-(\boldsymbol{y}|\boldsymbol{x}))$. 
\end{lemma}

\begin{proof}
By Lemma \ref{lem:exp}, we observe 
\begin{align*}
&(e^{y_1Y_1}\cdots e^{y_nY_n})(e^{s_1X_1}\cdots e^{s_nX_n})\\
&=(e^{y_1Y_1}e^{s_1X_1})(e^{y_2Y_2}e^{s_2X_2})\cdots (e^{y_nY_n}e^{s_nX_n})\\
&=(e^{s_1X_1}e^{y_1Y_1}e^{-y_1s_1Z})(e^{s_2X_2}e^{y_2Y_2}e^{-y_2s_2Z})\cdots 
	(e^{s_nX_n}e^{y_nY_n}e^{-y_ns_nZ})\\
&=(e^{s_1X_1}\cdots e^{s_nX_n})(e^{y_1Y_1}\cdots e^{y_nY_n})e^{-(y_1s_1+\cdots +y_ns_n)Z}. 
\end{align*}
Hence, Lemma \ref{lem:multiple} follows from the above equality. 
\end{proof}

Lemma \ref{lem:multiple} implies that we have a Lie group isomorphism 
\begin{align}
\label{eq:isom}
\iota :G\to \mathbb{R}^{2n}\ltimes \mathbb{R},\quad 
g(\boldsymbol{x},\boldsymbol{y},z)\mapsto (\boldsymbol{x},\boldsymbol{y},z). 
\end{align}

Next, 
any connected closed subgroup $H$ of $G$ is realized as $H=\exp \mathfrak{h}$ 
of some subalgebra $\mathfrak{h}$ of $\mathfrak{g}$. 
Hence, the following lemma is an immediate consequence of Lemma \ref{lem:subalgebra}. 

\begin{lemma}
\label{lem:subgroup}
A non-trivial connected closed subgroup of the Heisenberg  group $G$ 
is isomorphic to one of the following closed subgroups: 
\begin{enumerate}
	\renewcommand{\theenumi}{\roman{enumi}}
	\item $H_{(0,0,1)}:=Z(G)=\exp (\mathbb{R}Z)$ of $G$, 
	\label{gp:(0,0,1)}
	\item $H_{(p,q,1)}:=\exp \mathfrak{h}_{(p,q,1)}$ for some $1\leq p\leq q\leq n$, $(p,q)\neq (n,n)$, 
	\label{gp:(p,q,1)}
	\item $H_{(m,0,1)}:=\exp \mathfrak{h}_{(m,0,1)}$ for some $1\leq m\leq n$, 
	\label{gp:(m,0,1)}
	\item $H_{(m,0,0)}:=\exp \mathfrak{h}_{(m,0,0)}$ for some $1\leq m\leq n$. 
	\label{gp:(m,0,0)}
\end{enumerate}
\end{lemma}

We note that the closed subgroup $H_{(m,0,1)}$ in (\ref{gp:(m,0,1)}) is written by 
\begin{align*}
H_{(m,0,1)}=Z(G)H_{(m,0,0)}=H_{(m,0,0)}Z(G). 
\end{align*}

We put 
$\boldsymbol{0}_k=(0,\ldots ,0)\in \mathbb{R}^k$ $(1\leq k\leq n)$. 
Then, each closed subgroup given in Lemma \ref{lem:subgroup} forms as follows. 
\begin{align*}
H_{(0,0,1)}&=\{ g(\boldsymbol{0}_n,\boldsymbol{0}_n,z):z\in \mathbb{R}\} ,\\
H_{(p,q,1)}&=\{ g((\boldsymbol{w}_1,\boldsymbol{0}_{n-p}),
	(\boldsymbol{w}_2,\boldsymbol{0}_{n-q}),z):\boldsymbol{w}_1\in \mathbb{R}^p,
	\boldsymbol{w}_2\in \mathbb{R}^q,z\in \mathbb{R}\} ,\\
H_{(m,0,1)}&=\{ g((\boldsymbol{w},\boldsymbol{0}_{n-m}),
	\boldsymbol{0}_{n},z):\boldsymbol{w}\in \mathbb{R}^m,z\in \mathbb{R}\} ,\\
H_{(m,0,0)}&=\{ g((\boldsymbol{w},\boldsymbol{0}_{n-m}),
	\boldsymbol{0}_{n},0):\boldsymbol{w}\in \mathbb{R}^m\} .
\end{align*}

\subsection{The Lie group generated by $\exp \mathfrak{q}$}
\label{subsec:L}

Retain the notation as in Sections \ref{subsec:subalgebra} and \ref{subsec:subgroup}. 
For a subalgebra $\mathfrak{h}$ of $\mathfrak{g}$ and $\mathfrak{q}=C(\mathfrak{h})$, 
we denote by $\langle \exp \mathfrak{q}\rangle $ by the Lie group generated by $\exp \mathfrak{q}$. 
We set 
\begin{enumerate}
	\renewcommand{\theenumi}{\roman{enumi}}
	\item $Q_{(0,0,1)}:=\langle \exp \mathfrak{q}_{(0,0,1)}\rangle$. 
	\item $Q_{(p,q,1)}:=\langle \exp \mathfrak{q}_{(p,q,1)}\rangle$. 
	\item $Q_{(m,0,1)}:=\langle \exp \mathfrak{q}_{(m,0,1)}\rangle$. 
	\item $Q_{(m,0,0)}:=\langle \exp \mathfrak{q}_{(m,0,0)}\rangle$. 
\end{enumerate}

\begin{lemma}
\label{lem:Q}
The subgroup $Q=\langle \exp \mathfrak{q}\rangle$ of $G$ forms: 
\begin{enumerate}
	\renewcommand{\theenumi}{\roman{enumi}}
	\item $Q_{(0,0,1)}=G$. 
	\item 
	\begin{enumerate}
		\item If $1\leq p\leq q<n$ then $Q_{(p,q,1)}=\exp (\mathfrak{q}_{(p,q,1)}+\mathbb{R}Z)$. 
		\item If $1\leq p<q=n$ then $Q_{(p,n,1)}=\exp \mathfrak{q}_{(p,n,1)}$. 
	\end{enumerate}
	\item 
	\begin{enumerate}
		\item If $1\leq m<n$ then $Q_{(m,0,1)}=\exp (\mathfrak{q}_{(m,0,1)}+\mathbb{R}Z)$. 
		\item If $m=n$ then $Q_{(n,0,1)}=\exp \mathfrak{q}_{(n,0,1)}$. 
	\end{enumerate}
	\item $Q_{(m,0,0)}=\exp \mathfrak{q}_{(m,0,0)}$ for $1\leq m\leq n$. 
\end{enumerate}
\end{lemma}

\begin{proof}
We will prove this lemma using the bracket relations $[X_i,Y_i]=Z$ for $i=1,2,\ldots ,n$. 

First, $[\mathfrak{q}_{(0,0,1)},\mathfrak{q}_{(0,0,1)}]$ is contained in $\mathfrak{q}_{(0,0,1)}+\mathbb{R}Z$ 
which coincides with the Lie algebra $\mathfrak{g}$. 
Hence, $Q_{(0,0,1)}=\exp \mathfrak{g}=G$. 

Second, $\mathfrak{q}_{(p,q,1)}=\operatorname{span}_{\mathbb{R}}(\mathcal{B}_{(p,q,1)})^c$ 
contains $X_n$ and $Y_n$ if $q<n$. 
Then, the bracket $[\mathfrak{q}_{(p,q,1)},\mathfrak{q}_{(p,q,1)}]$ lies 
in $\mathfrak{q}_{(p,q,1)}+\mathbb{R}Z$, 
from which we obtain $Q_{(p,q,1)}=\exp (\mathfrak{q}_{(p,q,1)}+\mathbb{R}Z)$. 
On the other hand, 
if $p<q=n$, then the subspace 
$\mathfrak{q}_{(p,n,1)}=\operatorname{span}_{\mathbb{R}}\{ X_{p+1},\ldots ,X_n\} $ 
is an abelian subalgebra of $\mathfrak{g}$. 
Thus, we obtain $Q_{(p,n,1)}=\exp \mathfrak{q}_{(p,n,1)}$. 

Third, $\mathfrak{q}_{(m,0,1)}$ contains $X_n$ and $Y_n$. 
Then, we have 
$[\mathfrak{q}_{(m,0,1)},\mathfrak{q}_{(m,0,1)}]\subset \mathfrak{q}_{(m,0,1)}+\mathbb{R}Z$. 
Hence, we obtain $Q_{(m,0,1)}=\exp (\mathfrak{q}_{(m,0,1)}+\mathbb{R}Z)$. 
As $\mathfrak{q}_{(n,0,1)}$ is an abelian subalgebra, we have $Q_{(n,0,1)}=\exp \mathfrak{q}_{(n,0,1)}$. 

Finally, $\mathfrak{q}_{(m,0,0)}$ contains an element $Z$. 
Then, it is a subalgebra of $\mathfrak{g}$, from which $Q_{(m,0,0)}=\exp \mathfrak{q}_{(m,0,0)}$. 
\end{proof}

\subsection{Complex Heisenberg homogeneous spaces}
\label{subsec:homogeneous space}

Let $H$ be a connected closed subgroup of the Heisenberg group $G$ with Lie algebra $\mathfrak{h}$ 
which is one of (\ref{gp:(0,0,1)})--(\ref{gp:(m,0,0)}) given in Lemma \ref{lem:subgroup}. 
We denote by $\mathfrak{g}^{\mathbb{C}}$, $\mathfrak{h}^{\mathbb{C}}$ and $\mathfrak{q}^{\mathbb{C}}$ 
the complexifications of $\mathfrak{g}$, $\mathfrak{h}$ and $\mathfrak{q}$, respectively. 
Then, $\mathcal{B},\mathcal{B}_{\mathfrak{h}}$ and $(\mathcal{B}_{\mathfrak{h}})^c$ are 
$\mathbb{C}$-bases of 
$\mathfrak{g}^{\mathbb{C}}$, $\mathfrak{h}^{\mathbb{C}}$ and $\mathfrak{q}^{\mathbb{C}}$, respectively, 
and we have 
$\mathfrak{g}^{\mathbb{C}}=\mathfrak{h}^{\mathbb{C}}+\mathfrak{q}^{\mathbb{C}}$ 
and $\mathfrak{h}^{\mathbb{C}}\cap \mathfrak{q}^{\mathbb{C}}=\{ 0\} $. 

We set $G^{\mathbb{C}}:=\exp \mathfrak{g}^{\mathbb{C}}$ and $H^{\mathbb{C}}:=\exp \mathfrak{h}^{\mathbb{C}}$. 
Clearly, $G^{\mathbb{C}}$ is given by 
\begin{align*}
G^{\mathbb{C}}=\{ g(\boldsymbol{x},\boldsymbol{y},z):\boldsymbol{x},\boldsymbol{y}\in \mathbb{C}^n,
	z\in \mathbb{C}\} , 
\end{align*}
and hence it is isomorphic to $\mathbb{C}^{2n}\ltimes \mathbb{C}$ as a complex Lie group 
through $\iota :G^{\mathbb{C}}\to \mathbb{C}^{2n}\ltimes \mathbb{C}$ given by (\ref{eq:isom}). 

We say that $G^{\mathbb{C}}/H^{\mathbb{C}}$ is a {\it complex Heisenberg homogeneous space}. 
Clearly, this is of the form: 

\begin{lemma}
\label{lem:homogeneous}
The complex homogeneous space $G^{\mathbb{C}}/H^{\mathbb{C}}$ is given by 
\begin{align*}
G^{\mathbb{C}}/H^{\mathbb{C}}
=\{ (\exp X)H^{\mathbb{C}}:X\in \mathfrak{q}^{\mathbb{C}}\} 
\end{align*}
and biholomorphic to $\exp \mathfrak{q}^{\mathbb{C}}$. 
\end{lemma}

\subsection{Equivalent holomorphic actions}
\label{subsec:equivalent}

In this subsection, 
we will introduce the notion of equivalent holomorphic actions. 

Let $D_1,D_2$ be connected complex manifolds and $K_1,K_2$ be some Lie groups. 
Suppose that $K_1$ acts on $D_1$ and $K_2$ on $D_2$ holomorphically. 

\begin{define}
\label{def:equivalent}
We say that the holomorphic $K_1$-action on $D_1$ is {\it equivalent} 
to the holomorphic $K_2$-action on $D_2$ by $(\Phi ,\varphi)$ 
if there exists a biholomorphic diffeomorphism $\Phi :D_1\to D_2$ 
and a Lie group isomorphism $\varphi :K_1\to K_2$ such that 
\begin{align*}
\Phi (g\cdot v)=\varphi (g)\cdot \Phi (v)\quad (v\in D_1,g\in K_1). 
\end{align*}
\end{define}

Let us explain that 
the property of the strong visibility keeps among equivalent holomorphic actions. 
Namely, we prove: 

\begin{proposition}
\label{prop:equivalent}
Suppose that the holomorphic $K_1$-action on $D_1$ is equivalent 
to the holomorphic $K_2$-action on $D_2$ by $(\Phi ,\varphi )$. 
If the $K_1$-action on $D_1$ is strongly visible, 
then so is the $K_2$-action on $D_2$ and hence vice versa. 
\end{proposition}

\begin{proof}
By the assumption, 
one can take a real submanifold $S_1$ in $D_1$ such that $D_1':=K_1\cdot S_1$ is an open set in $D_1$, 
and take an anti-holomorphic diffeomorphism $\sigma _1$ on $D_1'$ 
such that $\sigma _1|_{S_1}=\operatorname{id}_{S_1}$ and 
$\sigma _1$ preserves each $K_1$-orbit in $D_1'$. 

We set $S_2:=\Phi (S_1)$ and $D_2':=K_2\cdot S_2$. 
Since $\Phi$ is a diffeomorphism as smooth manifolds, $S_2$ is a real submanifold in $D_2$. 
Further, 
the set $D_2'$ coincides with the image $\Phi (D_1')$ because 
\begin{align*}
\Phi (D_1')=\Phi (K_1\cdot S_1)=\varphi (K_1)\cdot \Phi (S_1)=K_2\cdot S_2. 
\end{align*}
As $D_1'$ is open in $D_1$, the image $\Phi (D_1')=D_2'$ is also open in $D_2$. 
Hence, $S_2$ satisfies the condition (\ref{v0}). 

Next, we define a diffeomorphism $\sigma _2 $ on $D$ by $\sigma _2 =\Phi \circ \sigma _1\circ \Phi ^{-1}$. 
This is an anti-holomorphic map since $\Phi$ is holomorphic. 

Here, we take an element $s_2\in S_2$. 
As $S_2=\Phi (S_1)$, we write $s_2=\Phi (s_1)$ for some $s_1\in S_1$. 
Since $\sigma _1(s_1)=s_1$, we have 
\begin{align*}
\sigma _2(s_2)
=\Phi \circ \sigma _1\circ \Phi ^{-1}(\Phi (s_1))
=\Phi \circ \sigma _1(s_1)
=\Phi (s_1)
=s_2. 
\end{align*}
This implies $\sigma _2$ is the identity map on $S_2$, 
from which the condition (\ref{s1}) has been verified. 

Finally, let us show that $\sigma _2$ satisfies the condition (\ref{s2}). 
For this, let $v_2$ be an element of $D_2'$. 
As $D_2'=\Phi (D_1')$, we write $v_2=\Phi (v_1)$ for some $v_1\in D_1'$. 
We recall that $\sigma _1(v_1)$ forms $\sigma _1(v_1)=g_1\cdot v_1$ for some $g_1\in K_1$. 
When we set $g_2:=\varphi (g_1)\in K_2$, we have 
\begin{align*}
\sigma _2(v_2)
=\Phi \circ \sigma _1\circ \Phi ^{-1}(\Phi (v_1))
=\Phi \circ \sigma _1(v_1)
=\Phi (g_1\cdot v_1)
=g_2\cdot v_2. 
\end{align*}
This means that $\sigma _2(v_2)=g_2\cdot v_2\in K_2\cdot v_2$. 

Therefore, Proposition \ref{prop:equivalent} has been proved. 
\end{proof}

Thanks to Proposition \ref{prop:equivalent}, 
we do not distinguish equivalent holomorphic actions in this paper. 

\subsection{Classification of actions on Heisenberg homogeneous spaces}
\label{subsec:classification}

In this subsection, 
we explain that it suffices to provide a proof of Theorem \ref{thm:main} 
in the case 
where a closed subgroup $H$ is taken to be one of (\ref{gp:(0,0,1)})--(\ref{gp:(m,0,0)}) 
in Lemma \ref{lem:subgroup}. 

First, let $\mathfrak{h}_1,\mathfrak{h}_2$ be some subalgebras of the Heisenberg Lie algebra $\mathfrak{g}$ 
satisfying $\mathfrak{h}_1\simeq \mathfrak{h}_2$. 

\begin{lemma}
\label{lem:automorphism}
If $\mathfrak{h}_1$ is isomorphic to $\mathfrak{h}_2$, 
then there exists an automorphism $\varphi $ on $\mathfrak{g}$ such that 
the restriction of $\varphi $ to $\mathfrak{h}_1$ gives rise to an isomorphism 
from $\mathfrak{h}_1$ to $\mathfrak{h}_2$. 
\end{lemma}

\begin{proof}
According to \cite[Proposition 3.1]{bky} (see also Lemma \ref{lem:subalgebra}), 
we may and do assume that a subalgebra $\mathfrak{h}_1$ is either 
$\mathfrak{h}_{(0,0,1)}=\mathbb{R}Z$, 
$\mathfrak{h}_{(p,q,1)}=\operatorname{span}_{\mathbb{R}}\mathcal{B}_{(p,q,1)}$, 
$\mathfrak{h}_{(m,0,1)}=\operatorname{span}_{\mathbb{R}}\mathcal{B}_{(m,0,1)}$ 
or $\mathfrak{h}_{(m,0,0)}=\operatorname{span}_{\mathbb{R}}\mathcal{B}_{(m,0,0)}$, 
and $\mathfrak{h}_2$ is isomorphic to $\mathfrak{h}_{(0,0,1)}$, 
$\mathfrak{h}_{(p,q,1)}$, $\mathfrak{h}_{(m,0,1)}$ or $\mathfrak{h}_{(m,0,0)}$. 

Let us treat the case $\mathfrak{h}_2\simeq \mathfrak{h}_{(p,q,1)}$. 
By \cite[Proposition 3.1]{bky}, 
there exists another basis 
$\mathcal{B}'=\{ X_1',\ldots ,X_n',Y_1',\ldots ,Y_n',Z'\} $ 
of $\mathfrak{g}$ with relations $[X_i',Y_i']=Z'$ ($1\leq i\leq n$) 
such that a basis of $\mathfrak{h}_2$ is taken to be the subset 
$\mathcal{B}_{\mathfrak{h}_2}=\mathcal{B}_{(p,q,1)}'
:=\{ X_1',\ldots ,X_p',Y_1',\ldots ,Y_q',Z'\} $ in $\mathcal{B}'$. 
Here, we define a linear transformation $\varphi $ on $\mathfrak{g}$ by 
$\varphi (X_i)=X_i',~\varphi (Y_i)=Y_i'$ ($1\leq i\leq n$) and $\varphi (Z)=Z'$. 
Then, $\varphi $ is an automorphism on $\mathfrak{g}$ and 
$\varphi :\mathfrak{h}_{(p,q,1)}\to \mathfrak{h}_2$ becomes a Lie algebra isomorphism. 

Similarly, we can also prove this lemma for other cases. 
Thus, we omit the proof. 
\end{proof}

Next, let $\varphi $ be an automorphism on $\mathfrak{g}$ such that 
$\varphi |_{\mathfrak{h}_1}:\mathfrak{h}_1\to \mathfrak{h}_2$ is a Lie algebra isomorphism. 
We extend $\varphi \in \operatorname{Aut}\mathfrak{g}$ to a $\mathbb{C}$-linear map 
on the complexification $\mathfrak{g}^{\mathbb{C}}$, namely, 
$\varphi (X+\sqrt{-1}Y)=\varphi (X)+\sqrt{-1}\varphi (Y)$ ($X,Y\in \mathfrak{g}$). 
Then, this is an automorphism on $\mathfrak{g}^{\mathbb{C}}$. 
Further, we can lift it to an automorphism 
on the complexified Heisenberg  group $G^{\mathbb{C}}=\exp \mathfrak{g}^{\mathbb{C}}$, 
for which we still denote by the  letter $\varphi $. 
Obviously, it is holomorphic. 

By Lemma \ref{lem:automorphism}, the image $\varphi (H_1^{\mathbb{C}})$ of 
$H_1^{\mathbb{C}}=\exp \mathfrak{h}_1^{\mathbb{C}}$ coincides 
with $H_2^{\mathbb{C}}=\exp \mathfrak{h}_2^{\mathbb{C}}$, from which 
it gives rise to a holomorphic diffeomorphism 
\begin{align}
\label{eq:Phi}
\Phi :G^{\mathbb{C}}/H_1^{\mathbb{C}}\to G^{\mathbb{C}}/H_2^{\mathbb{C}},\quad 
gH_1^{\mathbb{C}}\mapsto \Phi (gH_1^{\mathbb{C}})=\varphi (g)H_2^{\mathbb{C}}. 
\end{align}

Retain the notation as in the proof of Lemma \ref{lem:automorphism}. 
Let $\mathcal{B}_{\mathfrak{h}_1}\subset \mathcal{B}$ be a basis of $\mathfrak{h}_1$ 
and $\mathcal{B}_{\mathfrak{h}_2}\subset \mathcal{B}'$ a basis of $\mathfrak{h}_2$. 
We set $\mathfrak{q}_1=\operatorname{span}_{\mathbb{R}}(\mathcal{B}-\mathcal{B}_{\mathfrak{h}_1})$, 
$\mathfrak{q}_2=\operatorname{span}_{\mathbb{R}}(\mathcal{B}'-\mathcal{B}_{\mathfrak{h}_2})$, 
and $Q_1=\langle \exp \mathfrak{q}_1\rangle$, $Q_2=\langle \exp \mathfrak{q}_2\rangle$. 

\begin{lemma}
\label{lem:phi}
The Lie group $Q_1$ is isomorphic to $Q_2$ through $\varphi $. 
\end{lemma}

\begin{proof}
Clearly, 
$\varphi \in \operatorname{Aut}\mathfrak{g}$ 
gives a linear isomorphism from $\mathfrak{q}_1$ to $\mathfrak{q}_2$. 

Let us see that $\varphi (Q_1)$ is contained in $Q_2$ as follows. 
Let $A_1,\ldots ,A_k$ be elements of $\mathfrak{q}_1$. 
Then, $e^{A_1}\cdots e^{A_k}$ lies in $Q_1$. 
Since $\varphi $ is an automorphism on $G^{\mathbb{C}}$, 
we have $\varphi (e^{A_1}\cdots e^{A_k})=e^{\varphi (A_1)}\cdots e^{\varphi (A_k)}$. 
As $\varphi (A_1),\ldots ,\varphi (A_k)$ are in $\mathfrak{q}_2$, 
$\varphi (e^{A_1}\cdots e^{A_k})$ is an element of $Q_2$. 
Hence, we obtain $\varphi (Q_1)\subset Q_2$. 

On the other hand, we take an arbitrary $e^{B_1}\cdots e^{B_{\ell }}\in Q_2$ 
with $B_1,\ldots ,B_{\ell }\in \mathfrak{q}_2$. 
For each $1\leq i\leq \ell $, we set $A_i:=\varphi ^{-1}(B_i)$, which is an element of $\mathfrak{q}_1$. 
Then, we have $\varphi (e^{A_1}\cdots e^{A_{\ell }})=e^{B_1}\cdots e^{B_{\ell }}$. 
This means $\varphi (Q_1)\supset Q_2$. 

Consequently, we obtain $\varphi (Q_1)=Q_2$. 
Since $\varphi $ is an automorphism on $G^{\mathbb{C}}$, 
$\varphi :Q_1\to Q_2$ is a Lie group isomorphism. 
\end{proof}

\begin{proposition}
\label{prop:classification}
If $\mathfrak{h}_1$ is isomorphic to $\mathfrak{h}_2$ as a Lie algebra, 
then the holomorphic $Q_1$-action on $G^{\mathbb{C}}/H_1^{\mathbb{C}}$ is equivalent to 
the holomorphic $Q_2$-action on $G^{\mathbb{C}}/H_2^{\mathbb{C}}$. 
\end{proposition}

\begin{proof}
It follows from (\ref{eq:Phi}) that the following equality holds 
for $g\in Q_1$ and $x\in G^{\mathbb{C}}$: 
\begin{align*}
\Phi (g\cdot xH_1^{\mathbb{C}})
=\varphi (gx)H_2^{\mathbb{C}}
=\varphi (g)\cdot \varphi (x)H_2^{\mathbb{C}}
=\varphi (g)\cdot \Phi (xH_1^{\mathbb{C}}). 
\end{align*}
By Lemma \ref{lem:phi}, $\varphi (g)$ is an element of $Q_2$ for any $g\in Q_1$. 
Hence, the $Q_1$-action on $G^{\mathbb{C}}/H_1^{\mathbb{C}}$ is equivalent to 
the $Q_2$-action on $G^{\mathbb{C}}/H_2^{\mathbb{C}}$ by $(\Phi ,\varphi )$. 
\end{proof}

Retain the setting as in Proposition \ref{prop:classification} and its proof. 
Let $L_1$ be a closed subgroup of $G$ containing $Q_1$ and $L_2:=\varphi (L_1)\supset \varphi (Q_1)=Q_2$. 
The following corollary can be proved 
in the same argument as the proof of Proposition \ref{prop:classification}. 

\begin{corollary}
\label{cor:classification}
The holomorphic $L_1$-action on $G^{\mathbb{C}}/H_1^{\mathbb{C}}$ is equivalent to 
the holomorphic $L_2$-action on $G^{\mathbb{C}}/H_2^{\mathbb{C}}$ via $(\Phi ,\varphi )$. 
\end{corollary}

Due to Corollary \ref{cor:classification}, 
we shall deal with closed subgroups $H_{(k,\ell ,\varepsilon )}$ as in Lemma \ref{lem:subgroup} 
and write the Heisenberg homogeneous space $D_{(k,\ell ,\varepsilon )}$ as 
\begin{align}
\label{eq:D}
D_{(k,\ell ,\varepsilon )}
:=G^{\mathbb{C}}/H_{(k,\ell ,\varepsilon )}^{\mathbb{C}}
=\{ (\exp X)H_{(k,\ell ,\varepsilon )}^{\mathbb{C}}:X\in \mathfrak{q}_{(k,\ell ,\varepsilon )}^{\mathbb{C}}\} . 
\end{align}

\subsection{Anti-holomorphic diffeomorphisms on complex Heisenberg homogeneous spaces}
\label{subsec:diffeomorphism}

In this subsection, 
we will prepare two anti-holomorphic diffeomorphisms $\sigma _1,\sigma _2$ on 
the complex Heisenberg homogeneous space $D_{(k,\ell ,\varepsilon )}$. 

For a complex number $\alpha =\alpha _R+\sqrt{-1}\alpha _I\in \mathbb{C}$ 
with $\alpha _R,\alpha _I\in \mathbb{R}$, 
we set $\overline{\alpha }:=\alpha _R-\sqrt{-1}\alpha _I$. 
Equivalently, 
$\overline{\cdot }:\mathbb{C}\to \mathbb{C}$ is the complex conjugation of $\mathbb{C}$ 
with respect to the real form $\mathbb{R}$. 
We extend it to the complex conjugation of $\mathbb{C}^n$ 
with respect to the real form $\mathbb{R}^n$, namely, 
$\mathbb{C}^n\to \mathbb{C}^n,~\boldsymbol{x}\mapsto 
\overline{\boldsymbol{x}}:=(\overline{x_1},\ldots ,\overline{x_n})$. 

Now, 
we define an anti-holomorphic diffeomorphism $\widetilde{\sigma }_{1}$ on $G^{\mathbb{C}}$ by 
\begin{align}
\label{eq:involution1}
\widetilde{\sigma }_{1}(g(\boldsymbol{x},\boldsymbol{y},z))
:=g(-\overline{\boldsymbol{x}},-\overline{\boldsymbol{y}},\overline{z})\quad 
(g(\boldsymbol{x},\boldsymbol{y},z)\in G^{\mathbb{C}}), 
\end{align}
and an anti-holomorphic diffeomorphism $\widetilde{\sigma }_{2}$ on $G^{\mathbb{C}}$ by 
\begin{align}
\label{eq:involution2}
\widetilde{\sigma }_{2}(g(\boldsymbol{x},\boldsymbol{y},z))
:=g(\overline{\boldsymbol{x}},-\overline{\boldsymbol{y}},-\overline{z})\quad 
(g(\boldsymbol{x},\boldsymbol{y},z)\in G^{\mathbb{C}}). 
\end{align}

\begin{lemma}
\label{lem:involutive}
$\widetilde{\sigma }_1,\widetilde{\sigma }_2$ are involutive automorphisms on $G^{\mathbb{C}}$. 
\end{lemma}

\begin{proof}
Clearly, ${\widetilde{\sigma }_i}^2=\operatorname{id}$, from which $\widetilde{\sigma }_i$ is involutive 
for each $i=1,2$. 
In view of Lemma \ref{lem:multiple}, 
the direct computation shows 
\begin{align*}
&
\widetilde{\sigma}_1(g(\boldsymbol{x}_1,\boldsymbol{y}_1,z_1)\cdot g(\boldsymbol{x}_2,\boldsymbol{y}_2,z_2))
\\
&=g(-(\overline{\boldsymbol{x}_1+\boldsymbol{x}_2}),-(\overline{\boldsymbol{y}_1+\boldsymbol{y}_2}),
	\overline{z_1+z_2-(\boldsymbol{y}_1|\boldsymbol{x}_2)})\\
&=g(-\overline{\boldsymbol{x}_1}-\overline{\boldsymbol{x}_2},
	-\overline{\boldsymbol{y}_1}-\overline{\boldsymbol{y}_2},
	\overline{z_1}+\overline{z_2}-(\overline{\boldsymbol{y}_1}|\overline{\boldsymbol{x}_2}))\\
	&=g(-\overline{\boldsymbol{x}_1}-\overline{\boldsymbol{x}_2},
	-\overline{\boldsymbol{y}_1}-\overline{\boldsymbol{y}_2}, 
	\overline{z_1}+\overline{z_2}-(-\overline{\boldsymbol{y}_1}|-\overline{\boldsymbol{x}_2})\\ 
	&=g(-\overline{\boldsymbol{x}_1},-\overline{\boldsymbol{y}_1},\overline{z_1})\cdot 
	g(-\overline{\boldsymbol{x}_2},-\overline{\boldsymbol{y}_2},\overline{z_2})\\
	&=\widetilde{\sigma}_1(g(\boldsymbol{x}_1,\boldsymbol{y}_1,z_1))\cdot 
	\widetilde{\sigma}_1(g(\boldsymbol{x}_2,\boldsymbol{y}_2,z_2)).
\end{align*}
for any $g(\boldsymbol{s},\boldsymbol{t},u),g(\boldsymbol{x},\boldsymbol{y},z)\in G^{\mathbb{C}}$. 
This implies that $\widetilde{\sigma }_1$ is an automorphism on $G^{\mathbb{C}}$. 
Similarly, one can show that $\widetilde{\sigma }_2$ is also an automorphism on $G^{\mathbb{C}}$. 
\end{proof}

Let $H_{(k,\ell ,\varepsilon )}$ be a closed subgroup of $G$ 
which is one of (\ref{gp:(0,0,1)})--(\ref{gp:(m,0,0)}) in Lemma \ref{lem:subgroup} 
and $H_{(k,\ell ,\varepsilon )}^{\mathbb{C}}$ 
the complexification of $H_{(k,\ell ,\varepsilon )}$. 
Clearly, $\widetilde{\sigma }_i$ stablizes $H^{\mathbb{C}}_{(k,\ell ,\varepsilon )}$ for each $i=1,2$, 
which gives rise to an anti-holomorphic diffeomorphism $\sigma _i$ 
on the complex Heisenberg homogeneous space 
$D_{(k,\ell ,\varepsilon )}$, namely, 
\begin{align}
\label{eq:sigma}
\sigma _i((\exp X)H^{\mathbb{C}}_{(k,\ell ,\varepsilon )})
:=\widetilde{\sigma }_i(\exp X)H^{\mathbb{C}}_{(k,\ell ,\varepsilon )}\quad 
(X\in \mathfrak{q}_{(k,\ell ,\varepsilon )}^{\mathbb{C}}). 
\end{align}

\begin{remark}
\label{rem:compatible}
Concerning equality (\ref{eq:sigma}), 
the following equality holds: 
\begin{align*}
\sigma _i(g(\boldsymbol{s},\boldsymbol{t},u)\cdot (\exp X)H^{\mathbb{C}}_{(k,\ell ,\varepsilon )})
&=\widetilde{\sigma }_i(g(\boldsymbol{s},\boldsymbol{t},u)\cdot \exp X)
	H^{\mathbb{C}}_{(k,\ell ,\varepsilon )}\\
&=\widetilde{\sigma }_i(g(\boldsymbol{s},\boldsymbol{t},u))\cdot 
	\widetilde{\sigma }_i(\exp X)H^{\mathbb{C}}_{(k,\ell ,\varepsilon )}\\
&=\widetilde{\sigma }_i(g(\boldsymbol{s},\boldsymbol{t},u))\cdot 
	\sigma _i((\exp X)H^{\mathbb{C}}_{(k,\ell ,\varepsilon )}). 
\end{align*}
Hence, $\widetilde{\sigma }_i$ is compatible with $\sigma _i$ in the sense of \cite[Section 4.2]{ko13}. 
\end{remark}

\subsection{Real form of complex Heisenberg homogeneous space}
\label{subsec:real form}

In the previous subsection, 
we have prepared the anti-holomorphic diffeomorphisms $\sigma _1,\sigma _2$ on $D_{(k,\ell ,\varepsilon )}$ 
defined by (\ref{eq:sigma}). 
Then, we have a real form $M_{(k,\ell ,\varepsilon )}^i$ of $D_{(k,\ell ,\varepsilon )}$ 
corresponding to $\sigma _i$, 
namely, 
\begin{align*}
M_{(k,\ell ,\varepsilon )}^i
:=D_{(k,\ell ,\varepsilon )}^{\sigma _i}
=\{ v\in D_{(k,\ell ,\varepsilon )}:\sigma _i(v)=v\} 
\end{align*}
for each $i=1,2$. 
To clarify $M_{(k,\ell ,\varepsilon )}^i$, we observe: 

\begin{lemma}
\label{lem:real form}
For an element $v=(\exp X)H_{(k,\ell ,\varepsilon )}^{\mathbb{C}}$ 
with $X\in \mathfrak{q}_{(k,\ell ,\varepsilon )}^{\mathbb{C}}$, 
the equality $\sigma _i(v)=v$ holds if and only if $\widetilde{\sigma }_i(\exp X)=\exp X$ holds. 
\end{lemma}

\begin{proof}
If $X\in \mathfrak{q}_{(k,\ell ,\varepsilon )}^{\mathbb{C}}$ satisfies 
$\widetilde{\sigma }_i(\exp X)=\exp X$, 
then we have 
\begin{align*}
\sigma _i((\exp X)H_{(k,\ell ,\varepsilon )}^{\mathbb{C}})
=\widetilde{\sigma }_i(\exp X)H_{(k,\ell ,\varepsilon )}^{\mathbb{C}}
=(\exp X)H_{(k,\ell ,\varepsilon )}^{\mathbb{C}}. 
\end{align*}
This means $\sigma _i(v)=v$. 

Conversely, suppose that $v$ satisfies $\sigma _i(v)=v$. 
Then, $\exp (-X)\widetilde{\sigma }_i(\exp X)$ lies in $H_{(k,\ell ,\varepsilon )}^{\mathbb{C}}$. 
Since 
$\mathfrak{h}_{(k,\ell ,\varepsilon )}^{\mathbb{C}}\cap \mathfrak{q}_{(k,\ell ,\varepsilon )}^{\mathbb{C}}
=\{ 0\} $, 
it has to be the unit element of $G^{\mathbb{C}}$. 
Thus, we get $\widetilde{\sigma }_i(\exp X)=\exp X$. 
\end{proof}

By Lemma \ref{lem:real form}, 
our choice of real form $M_{(k,\ell ,\varepsilon )}^i$ forms 
\begin{align}
\label{eq:real form}
M_{(k,\ell ,\varepsilon )}^i
=\{ (\exp X)H_{(k,\ell ,\varepsilon )}^{\mathbb{C}}:
X\in \mathfrak{q}_{(k,\ell ,\varepsilon )}^{\mathbb{C}},\widetilde{\sigma }_i(\exp X)=\exp X\} . 
\end{align}

\subsection{Strategy of proof of Theorem \ref{thm:main}}
\label{subsec:strategy}

In this subsection, 
we will explain how to prove Theorem \ref{thm:main}. 
More precisely, 
our proof is based on the following lemma. 

\begin{lemma}
\label{lem:reduction}
Suppose that for the anti-holomorphic involution $\widetilde{\sigma }_i$ on $G^{\mathbb{C}}$ 
there exist a $\widetilde{\sigma }_i$-stable connected closed subgroup 
$L_{(k,\ell ,\varepsilon )}$ of $G$ containing $Q_{(k,\ell ,\varepsilon )}$ and 
a submanifold $S_{(k,\ell ,\varepsilon )}$ in $M_{(k,\ell ,\varepsilon )}^i$ such that 
$D_{(k,\ell ,\varepsilon )}':=L_{(k,\ell ,\varepsilon )}\cdot S_{(k,\ell ,\varepsilon )}$ 
is an open set in $D_{(k,\ell ,\varepsilon )}$. 
Then, $S_{(k,\ell ,\varepsilon )}$ and the anti-holomorphic diffeomorphism $\sigma _i$ 
on $D_{(k,\ell ,\varepsilon )}$ 
satisfy the conditions (\ref{s1}) and (\ref{s2}) 
for the $L_{(k,\ell ,\varepsilon )}$-action on $D_{(k,\ell ,\varepsilon )}$. 
\end{lemma}

\begin{proof}
As $S_{(k,\ell ,\varepsilon )}\subset M_{(k,\ell ,\varepsilon )}^i=D_{(k,\ell ,\varepsilon )}^{\sigma _i}$, 
we have $\sigma _i|_{S_{(k,\ell ,\varepsilon )}}
=\operatorname{id}_{S_{(k,\ell ,\varepsilon )}}$, 
from which (\ref{s1}) has been verified. 

Next, let us take an element $v=g\cdot s\in D_{(k,\ell ,\varepsilon )}'$ 
with $g\in L_{(k,\ell ,\varepsilon )}$ and $s\in S_{(k,\ell ,\varepsilon )}$. 
Then, we have 
\begin{align}
\label{eq:s2}
\sigma _i(v)
=\widetilde{\sigma }_i(g)\cdot \sigma _i(s)
=\widetilde{\sigma }_i(g)\cdot s
=\widetilde{\sigma }_i(g)g^{-1}\cdot v. 
\end{align}
Since $L_{(k,\ell ,\varepsilon )}$ is $\widetilde{\sigma }_i$-stable, 
the element $\widetilde{\sigma }_i(g)g^{-1}$ lies in $L_{(k,\ell ,\varepsilon )}$. 
Hence, (\ref{eq:s2}) implies condition (\ref{s2}). 
\end{proof}

Thanks to Lemma \ref{lem:reduction}, 
it is sufficient for the proof of Theorem \ref{thm:main} 
to find suitable $L_{(k,\ell ,\varepsilon )}$ and $S_{(k,\ell ,\varepsilon )}$ 
such that $L_{(k,\ell ,\varepsilon )}\cdot S_{(k,\ell ,\varepsilon )}$ is open in $D_{(k,\ell ,\varepsilon )}$. 
Then, we will give an explicit description of such pair 
$(L_{(k,\ell ,\varepsilon )},S_{(k,\ell ,\varepsilon )})$ 
for $D_{(k,\ell ,\varepsilon )}=D_{(0,0,1)}$ in Section \ref{sec:(0,0,1)}; 
for $D_{(k,\ell ,\varepsilon )}=D_{(p,q,1)}$ in Section \ref{sec:(p,q,1)}; 
for $D_{(k,\ell ,\varepsilon )}=D_{(m,0,1)}$ in Section \ref{sec:(m,0,1)}; 
and for $D_{(k,\ell ,\varepsilon )}=D_{(m,0,0)}$ in Section \ref{sec:(m,0,0)}.


\section{Visible actions on $G^{\mathbb{C}}/H_{(0,0,1)}^{\mathbb{C}}$}
\label{sec:(0,0,1)}

This section deals with the $Q_{(0,0,1)}$-action 
on $D_{(0,0,1)}=G^{\mathbb{C}}/H_{(0,0,1)}^{\mathbb{C}}$. 
Recall that  
$\mathfrak{h}_{(0,0,1)}=\mathfrak{z}(\mathfrak{g})=\mathbb{R}Z$ is the center of $\mathfrak{g}$. 
Then, $H_{(0,0,1)}^{\mathbb{C}}=\exp \mathfrak{h}_{(0,0,1)}^{\mathbb{C}}
=\exp \mathbb{C}Z$ and 
\begin{align*}
\mathfrak{q}_{(0,0,1)}=\operatorname{span}_{\mathbb{R}}\{ X_1,\ldots ,X_n,Y_1,\ldots ,Y_n\} . 
\end{align*}
Thus, it follows from Lemma \ref{lem:homogeneous} that 
$D_{(0,0,1)}:=G^{\mathbb{C}}/H^{\mathbb{C}}_{(0,0,1)}$ is written as follows: 
\begin{align*}
D_{(0,0,1)}
=\{ v_{(0,0,1)}(\boldsymbol{x},\boldsymbol{y})
	:=g(\boldsymbol{x},\boldsymbol{y},0)H^{\mathbb{C}}_{(0,0,1)}: 
	\boldsymbol{x},\boldsymbol{y}\in \mathbb{C}^n\} . 
\end{align*}

Now, we consider the action of $Q_{(0,0,1)}$ on $D_{(0,0,1)}$. 
By Lemma \ref{lem:Q}, $Q_{(0,0,1)}$ coincides with $G={\bf{H}}_n$. 
In view of Lemma \ref{lem:multiple}, 
the $Q_{(0,0,1)}$-action on $D_{(0,0,1)}$ is given by 
\begin{align*}
g(\boldsymbol{s},\boldsymbol{t},u)\cdot v_{(0,0,1)}(\boldsymbol{x},\boldsymbol{y})
&=(g(\boldsymbol{s},\boldsymbol{t},u)\cdot g(\boldsymbol{x},\boldsymbol{y},0))H^{\mathbb{C}}_{(0,0,1)}\\
&=g(\boldsymbol{s}+\boldsymbol{x},\boldsymbol{t}+\boldsymbol{y},u-(\boldsymbol{t}|\boldsymbol{x}))
	H_{(0,0,1)}^{\mathbb{C}}\\
&=g(\boldsymbol{s}+\boldsymbol{x},\boldsymbol{t}+\boldsymbol{y},0)H_{(0,0,1)}^{\mathbb{C}}\\
&=v_{(0,0,1)}(\boldsymbol{s}+\boldsymbol{x},\boldsymbol{t}+\boldsymbol{y}). 
\end{align*}

Now, we take an anti-holomorphic diffeomorphism on $D_{(0,0,1)}$ as $\sigma _1$ (see (\ref{eq:sigma})). 
By Lemma \ref{lem:real form}, 
our choice of the real form $M_{(0,0,1)}^1=D_{(0,0,1)}^{\sigma _1}$ is written by 
\begin{align}
\label{eq:slice-(0,0,1)}
M_{(0,0,1)}^1
=\{ v_{(0,0,1)}(\sqrt{-1}\boldsymbol{a},\sqrt{-1}\boldsymbol{b})
	:\boldsymbol{a},\boldsymbol{b}\in \mathbb{R}^{n}\} 
\end{align}
since $\widetilde{\sigma }_1(g(\boldsymbol{x},\boldsymbol{y},0))
=g(-\overline{\boldsymbol{x}},-\overline{\boldsymbol{y}},0)$ 
for $g(\boldsymbol{x},\boldsymbol{y},0)\in \exp \mathfrak{q}_{(0,0,1)}^{\mathbb{C}}$. 

\begin{proposition}
\label{prop:v0-(0,0,1)}
We take a connected closed subgroup $L_{(0,0,1)}$ as $Q_{(0,0,1)}$ and 
a submanifold $S_{(0,0,1)}$ as $M_{(0,0,1)}^1$. 
Then, $L_{(0,0,1)}\cdot S_{(0,0,1)}$ coincides with $D_{(0,0,1)}$. 
\end{proposition}

\begin{proof}
Let $v_{(0,0,1)}(\boldsymbol{x},\boldsymbol{y})$ be an element of $D_{(0,0,1)}$. 
According to the decomposition $\mathbb{C}^n=\mathbb{R}^n+\sqrt{-1}\mathbb{R}^n$, 
we write $\boldsymbol{x}=\boldsymbol{x}_R+\sqrt{-1}\boldsymbol{x}_I$ 
for some $\boldsymbol{x}_R,\boldsymbol{x}_I\in \mathbb{R}^{2n}$ 
and similarly $\boldsymbol{y}=\boldsymbol{y}_R+\sqrt{-1}\boldsymbol{y}_I$ 
for some $\boldsymbol{y}_R,\boldsymbol{y}_I\in \mathbb{R}^{2n}$. 
Then, we have 
\begin{align}
\label{eq:v0-(0,0,1)}
v_{(0,0,1)}(\boldsymbol{x},\boldsymbol{y})
=g(\boldsymbol{x}_R,\boldsymbol{y}_R,0)\cdot v(\sqrt{-1}\boldsymbol{x}_I,\sqrt{-1}\boldsymbol{y}_I). 
\end{align}
Since $g(\boldsymbol{x}_R,\boldsymbol{y}_R,0)\in Q_{(0,0,1)}$ and 
$v_{(0,0,1)}(\sqrt{-1}\boldsymbol{x}_I,\sqrt{-1}\boldsymbol{y}_I)\in S_{(0,0,1)}$, 
this implies $v_{(0,0,1)}(\boldsymbol{x},\boldsymbol{y})\in Q_{(0,0,1)}\cdot S_{(0,0,1)}$. 
Hence, Proposition \ref{prop:v0-(0,0,1)} has been verified. 
\end{proof}

Therefore, we conclude: 

\begin{theorem}
\label{thm:visible-(0,0,1)}
The $Q_{(0,0,1)}$-action on $D_{(0,0,1)}=G^{\mathbb{C}}/H_{(0,0,1)}^{\mathbb{C}}$ 
is strongly visible with $2n$-dimensional slice 
$M_{(0,0,1)}^1\simeq \exp \sqrt{-1}\mathfrak{q}_{(0,0,1)}$. 
\end{theorem}


\section{Visible actions on $G^{\mathbb{C}}/H_{(p,q,1)}^{\mathbb{C}}$}
\label{sec:(p,q,1)}

This section deals with the case where a closed subgroup of $G$ is $H_{(p,q,1)}$ 
for $1\leq p\leq q\leq n$, $(p,q)\neq (n,n)$. 
We set $D_{(p,q,1)}:=G^{\mathbb{C}}/H_{(p,q,1)}^{\mathbb{C}}$. 
From now, let us consider the $Q_{(p,q,1)}$-action on $D_{(p,q,1)}$. 

In view of Lemma \ref{lem:q}, 
we study separately the following dichotomous cases: a general case $1\leq p\leq q<n$; 
a special case $1\leq p<q=n$. 

\subsection{$Q_{(p,q,1)}$-action on $D_{(p,q,1)}$}
\label{subsec:action-(p,q,1)}

First, we consider a general case $1\leq p\leq q<n$. 
The complementary subspace $\mathfrak{q}_{(p,q,1)}$ of $\mathfrak{h}_{(p,q,1)}$ in $\mathfrak{g}$ 
is given by 
\begin{align*}
\mathfrak{q}_{(p,q,1)}
=\operatorname{span}_{\mathbb{R}}\{ X_{p+1},\ldots ,X_n,Y_{q+1},\ldots ,Y_n\} . 
\end{align*}
By Lemma \ref{lem:Q}, the Lie group $Q_{(p,q,1)}=\langle \exp \mathfrak{q}_{(p,q,1)}\rangle $ 
equals $\exp (\mathfrak{q}_{(p,q,1)}+\mathbb{R}Z)$, which is of the form 
\begin{align*}
Q_{(p,q,1)}=\{ g((\boldsymbol{0}_p,\boldsymbol{s}'),(\boldsymbol{0}_q,\boldsymbol{t}'),u):
	\boldsymbol{s}'
	\in \mathbb{R}^{n-p},
	\boldsymbol{t}'\in \mathbb{R}^{n-q},
	u\in \mathbb{R}\} . 
\end{align*}
We note that $Q_{(p,q,1)}$ is isomorphic to 
$\mathbb{R}^{q-p}\times {\bf{H}}_{n-q}$ as a Lie group if $p<q$ and 
${\bf{H}}_{n-q}$ if $p=q$. 
Further, the complex Heisenberg homogeneous space $D_{(p,q,1)}$ is expressed by 
\begin{align*}
D_{(p,q,1)}=\{ v_{(p,q,1)}(\boldsymbol{x}',\boldsymbol{y}'):
	\boldsymbol{x}'\in \mathbb{C}^{n-p},\boldsymbol{y}'\in \mathbb{C}^{n-q}\} . 
\end{align*}
where we write 
\begin{align*}
v_{(p,q,1)}(\boldsymbol{x}',\boldsymbol{y}')
:=g((\boldsymbol{0}_p,\boldsymbol{x}'),(\boldsymbol{0}_q,\boldsymbol{y}'),0)
	H_{(p,q,1)}^{\mathbb{C}}. 
\end{align*}
Then, it follows from Lemma \ref{lem:multiple} that 
the $Q_{(p,q,1)}$-action on $D_{(p,q,1)}$ is given by 
\begin{align*}
&g((\boldsymbol{0}_p,\boldsymbol{s}'),(\boldsymbol{0}_q,\boldsymbol{t}'),u)\cdot 
	v_{(p,q,1)}(\boldsymbol{x}',\boldsymbol{y}')\\
&=(g((\boldsymbol{0}_p,\boldsymbol{s}'),(\boldsymbol{0}_q,\boldsymbol{t}'),u)
	\cdot g((\boldsymbol{0}_p,\boldsymbol{x}'),(\boldsymbol{0}_q,\boldsymbol{y}'),0))
	H_{(p,q,1)}^{\mathbb{C}}\\
&=g((\boldsymbol{0}_p,\boldsymbol{s}'+\boldsymbol{x}'),
	(\boldsymbol{0}_q,\boldsymbol{t}'+\boldsymbol{y}'),
	u-((\boldsymbol{0}_q,\boldsymbol{t}')|(\boldsymbol{0}_p,\boldsymbol{x}')))H_{(p,q,1)}^{\mathbb{C}}\\
&=g((\boldsymbol{0}_p,\boldsymbol{s}'+\boldsymbol{x}'),
	(\boldsymbol{0}_q,\boldsymbol{t}'+\boldsymbol{y}'),0)H_{(p,q,1)}^{\mathbb{C}}\\
&=v_{(p,q,1)}(\boldsymbol{s}'+\boldsymbol{x}',\boldsymbol{t}'+\boldsymbol{y}') 
\end{align*}
for $g((\boldsymbol{0}_p,\boldsymbol{s}'),(\boldsymbol{0}_q,\boldsymbol{t}'),u)\in Q_{(p,q,1)}$ 
and $v_{(p,q,1)}(\boldsymbol{x}',\boldsymbol{y}')\in D_{(p,q,1)}$. 

In the case $1\leq p<q=n$, 
the group $Q_{(p,n,1)}$ is given by 
\begin{align*}
Q_{(p,n,1)}=\{ g((\boldsymbol{0}_p,\boldsymbol{s}'),\boldsymbol{0}_n,0):
	\boldsymbol{s}'\in \mathbb{R}^{n-p},u\in \mathbb{R}\} 
\end{align*}
and the Heisenberg homogeneous space $D_{(p,n,1)}$ by 
\begin{align*}
D_{(p,n,1)}=\{ v_{(p,n,1)}(\boldsymbol{x}')
	:=g((\boldsymbol{0}_p,\boldsymbol{x}'),\boldsymbol{0}_n,0)H_{(p,n,1)}^{\mathbb{C}}:
	\boldsymbol{x}'\in \mathbb{C}^{n-p}\} . 
\end{align*}
Then, the $Q_{(p,n,1)}$-action on $D_{(p,n,1)}$ is written by 
\begin{align*}
g((\boldsymbol{0}_p,\boldsymbol{s}'),\boldsymbol{0}_n,0)\cdot v_{(p,n,1)}(\boldsymbol{x}')
=v_{(p,n,1)}(\boldsymbol{s}'+\boldsymbol{x}')
\end{align*}
for $g((\boldsymbol{0}_p,\boldsymbol{s}'),\boldsymbol{0}_n,0)\in Q_{(p,n,1)}$ and 
$v_{(p,n,1)}(\boldsymbol{x}')\in D_{(p,n,1)}$. 

\subsection{Verification of (\ref{v0}) for $D_{(p,q,1)}$}
\label{subsec:v0-(p,q,1)}

Let us take $L_{(p,q,1)}$ as $Q_{(p,q,1)}$ and 
an anti-holomorphic diffeomorphism on $D_{(p,q,1)}$ as $\sigma _1$. 
Then, we prove: 

\begin{proposition}
\label{prop:v0-(p,q,1)}
We take a connected closed subgroup $L_{(p,q,1)}$ as $Q_{(p,q,1)}$ 
and a submanifold $S_{(p,q,1)}$ in $D_{(p,q,1)}$ as $M_{(p,q,1)}^1$. 
Then, $L_{(p,q,1)}\cdot S_{(p,q,1)}$ coincides with $D_{(p,q,1)}$. 
\end{proposition}

First, let us assume that $1\leq p\leq q<n$. 
We note that 
the restriction of the anti-holomorphic involution $\widetilde{\sigma }_1$ 
to $\exp \mathfrak{q}_{(p,q,1)}^{\mathbb{C}}$ is 
\begin{align*}
\widetilde{\sigma }_1(g((\boldsymbol{0}_p,\boldsymbol{x}'),(\boldsymbol{0}_q,\boldsymbol{y}'),0))
=g((\boldsymbol{0}_p,-\overline{\boldsymbol{x}'}),(\boldsymbol{0}_q,-\overline{\boldsymbol{y}'}),0) 
\end{align*}
for $g((\boldsymbol{0}_p,\boldsymbol{x}'),(\boldsymbol{0}_q,\boldsymbol{y}'),0)
\in \exp \mathfrak{q}_{(p,q,1)}^{\mathbb{C}}$. 
It follows from Lemma \ref{lem:real form} that 
the real form $M_{(p,q,1)}^1$ is given by 
\begin{align}
\label{eq:slice-(p,q,1)}
M_{(p,q,1)}^1=\{ v_{(p,q,1)}(\sqrt{-1}\boldsymbol{a}',\sqrt{-1}\boldsymbol{b}'):
	\boldsymbol{a}'\in \mathbb{R}^{n-p},\boldsymbol{b}\in \mathbb{R}^{n-q}\} . 
\end{align}

\begin{proof}[Proof of Proposition \ref{prop:v0-(p,q,1)} for $1\leq p\leq q<n$]
Let $v_{(p,q,1)}(\boldsymbol{x}',\boldsymbol{y}')$ be an element of $D_{(p,q,1)}$. 
We write $\boldsymbol{x}'=\boldsymbol{x}'_R+\sqrt{-1}\boldsymbol{x}'_I$ and 
$\boldsymbol{y}'=\boldsymbol{y}'_R+\sqrt{-1}\boldsymbol{y}'_I$ 
for some $\boldsymbol{x}'_R,\boldsymbol{y}'_R\in \mathbb{R}^{n-p}$ and 
$\boldsymbol{x}'_I,\boldsymbol{y}'_I\in \mathbb{R}^{n-q}$. 
Then, we have 
\begin{align*}
v_{(p,q,1)}(\boldsymbol{x}',\boldsymbol{y}')
=g((\boldsymbol{0}_p,\boldsymbol{x}'_R),(\boldsymbol{0}_q,\boldsymbol{y}'_R),0)\cdot 
	v_{(p,q,1)}(\sqrt{-1}\boldsymbol{x}'_I,\sqrt{-1}\boldsymbol{y}'_I). 
\end{align*}
As $g((\boldsymbol{0}_p,\boldsymbol{x}'_R),(\boldsymbol{0}_q,\boldsymbol{y}'_R),0)\in Q_{(p,q,1)}$ 
and $v_{(p,q,1)}(\sqrt{-1}\boldsymbol{x}'_I,\sqrt{-1}\boldsymbol{y}'_I)\in S_{(p,q,1)}$, 
this implies that $Q_{(p,q,1)}\cdot S_{(p,q,1)}=D_{(p,q,1)}$. 
\end{proof}

For a special case $1\leq p<q=n$, 
the real form $M_{(p,n,1)}^1$ is 
\begin{align}
\label{eq:slice-(p,n.1)}
M_{(p,n,1)}^1=\{ v_{(p,n,1)}(\sqrt{-1}\boldsymbol{a}'):\boldsymbol{a}'\in \mathbb{R}^{n-p}\} . 
\end{align}

\begin{proof}[Proof of Proposition \ref{prop:v0-(p,q,1)} for $1\leq p<q=n$]
An element $v_{(p,n,1)}(\boldsymbol{x}')$ of $D_{(p,n,1)}$ 
with $\boldsymbol{x}'=\boldsymbol{x}_R'+\sqrt{-1}\boldsymbol{x}_I'$ 
($\boldsymbol{x}_R',\boldsymbol{x}_I'\in \mathbb{R}^{n-p}$) is of the form 
\begin{align*}
v_{(p,n,1)}(\boldsymbol{x}')
=g((\boldsymbol{0}_p,\boldsymbol{x}'_R),\boldsymbol{0}_n,0)
\cdot v_{(p,n,1)}(\sqrt{-1}\boldsymbol{x}_I')\in Q_{(p,n,1)}\cdot S_{(p,n,1)}. 
\end{align*}
Hence, we have verified $D_{(p,n,1)}=Q_{(p,n,1)}\cdot S_{(p,n,1)}$. 
\end{proof}

Hence, we conclude: 

\begin{theorem}
\label{thm:visible-(p,q,1)}
For $1\leq p\leq q\leq n$, $(p,q)\neq (n,n)$, 
the $Q_{(p,q,1)}$-action on $D_{(p,q,1)}=G^{\mathbb{C}}/H_{(p,q,1)}^{\mathbb{C}}$ 
is strongly visible with $(2n-p-q)$-dimensional slice 
$M_{(p,q,1)}^1\simeq \exp \sqrt{-1}\mathfrak{q}_{(p,q,1)}$. 
\end{theorem}


\section{Visible action on $G^{\mathbb{C}}/H_{(m,0,1)}^{\mathbb{C}}$}
\label{sec:(m,0,1)}

This section considers the Heisenberg homogeneous space 
$D_{(m,0,1)}=G^{\mathbb{C}}/H^{\mathbb{C}}_{(m,0,1)}$ with $1\leq m\leq n$. 

\subsection{$Q_{(m,0,1)}$-action on $D_{(m,0,1)}$}
\label{subsec:action-(m,0,1)}

Now, we let $m$ be $1\leq m<n$. 
The complementary subspace $\mathfrak{q}_{(m,0,1)}$ of $\mathfrak{h}_{(m,0,1)}$ is 
\begin{align*}
\mathfrak{q}_{(m,0,1)}
=\operatorname{span}_{\mathbb{R}}\{ X_{m+1},\ldots ,X_n,Y_1,\ldots ,Y_n\} . 
\end{align*}
Then, $D_{(m,0,1)}$ is of the form 
\begin{align*}
D_{(m,0,1)}=\{ g((\boldsymbol{0}_m,\boldsymbol{x}'),\boldsymbol{y},0)H_{(m,0,1)}^{\mathbb{C}}:
	\boldsymbol{x}'\in \mathbb{C}^{n-m},\boldsymbol{y}\in \mathbb{C}^n\}  
\end{align*}
where 
\begin{align*}
v_{(m,0,1)}(\boldsymbol{x}',\boldsymbol{y})
:=g((\boldsymbol{0}_m,\boldsymbol{x}'),\boldsymbol{y},0)H_{(m,0,1)}^{\mathbb{C}}. 
\end{align*}
Here, $\mathfrak{q}_{(m,0,1)}$ is not a subalgebra. 
It follows from Lemma \ref{lem:Q} that 
the subgroup $Q_{(m,0,1)}=\langle \exp \mathfrak{q}_{(m,0,1)}\rangle$ is given by 
\begin{align*}
Q_{(m,0,1)}=\{ g((\boldsymbol{0}_m,\boldsymbol{s}'),\boldsymbol{t},u):
	\boldsymbol{s}'\in \mathbb{R}^{n-m},\boldsymbol{t}\in \mathbb{R}^n,u\in \mathbb{R}\} . 
\end{align*}
Thus, $Q_{(m,0,1)}$ is isomorphic to $\mathbb{R}^m\times {\bf{H}}_{n-m}$ and acts on $D_{(m,0,1)}$ by 
\begin{align*}
&g((\boldsymbol{0}_m,\boldsymbol{s}'),\boldsymbol{t},u)\cdot v_{(m,0,1)}(\boldsymbol{x}',\boldsymbol{y})\\
&=g((\boldsymbol{0}_m,\boldsymbol{s}'+\boldsymbol{x}'),\boldsymbol{t}+\boldsymbol{y},
	u-(\boldsymbol{t}|(\boldsymbol{0}_m,\boldsymbol{x}'))H_{(m,0,1)}^{\mathbb{C}}\\
&=g((\boldsymbol{0}_m,\boldsymbol{s}'+\boldsymbol{x}'),\boldsymbol{t}+\boldsymbol{y})
	H_{(m,0,1)}^{\mathbb{C}}\\
&=v_{(m,0,1)}(\boldsymbol{s}'+\boldsymbol{x}',\boldsymbol{t}+\boldsymbol{y}). 
\end{align*}

For $m=n$, we have $\mathfrak{q}_{(n,0,1)}=\operatorname{span}_{\mathbb{R}}\{ Y_1,\ldots ,Y_n\} $. 
Then, 
we obtain $Q_{(n,0,1)}=\{ g(\boldsymbol{0}_n,\boldsymbol{t},0):\boldsymbol{t}\in \mathbb{R}^n\} 
\simeq \mathbb{R}^n$ 
and $D_{(n,0,1)}=\{ v_{(n,0,1)}(\boldsymbol{y}):=g(\boldsymbol{0}_n,\boldsymbol{y},0)H_{(n,0,1)}^{\mathbb{C}}:
\boldsymbol{y}\in \mathbb{C}^n\}$. 
Further, $Q_{(n,0,1)}$ acts on $D_{(n,0,1)}$ by 
\begin{align*}
g(\boldsymbol{0}_n,\boldsymbol{t},0)\cdot v_{(n,0,1)}(\boldsymbol{y})
=v_{(n,0,1)}(\boldsymbol{t}+\boldsymbol{y})
\end{align*}
for $g(\boldsymbol{0}_n,\boldsymbol{t},0)\in Q_{(n,0,1)}$ and 
$v_{(n,0,1)}(\boldsymbol{y})\in D_{(n,0,1)}$. 

\subsection{Verification of (\ref{v0}) for $D_{(m,0,1)}$}
\label{subsec:v0-(m,0,1)}

First, we consider the case $m<n$. 
Let us take an anti-holomorphic diffeomorphism on $D_{(m,0,1)}$ as $\sigma _1$ (see (\ref{eq:sigma})). 
We recall that the restriction of the anti-holomorphic involution 
$\widetilde{\sigma }_1$ on $G^{\mathbb{C}}$ (see (\ref{eq:involution1})) 
to $\exp \mathfrak{q}_{(m,0,1)}^{\mathbb{C}}$ is given by 
\begin{align*}
\widetilde{\sigma }_1(g((\boldsymbol{0}_m,\boldsymbol{x}'),\boldsymbol{y},0))
=g((\boldsymbol{0}_m,-\overline{\boldsymbol{x}'}),-\overline{\boldsymbol{y}},0). 
\end{align*}
By (\ref{eq:real form}), the real form $M_{(m,0,1)}^1=D_{(m,0,1)}^{\sigma _1}$ is given as follows: 
\begin{align}
\label{eq:slice-(m,0,1)}
M_{(m,0,1)}^1
&=\{ v_{(m,0,1)}(\sqrt{-1}\boldsymbol{a}',\sqrt{-1}\boldsymbol{b}):
	\boldsymbol{a}'\in \mathbb{R}^{n-m},\boldsymbol{b}\in \mathbb{R}^n\} . 
\end{align}

\begin{proposition}
\label{prop:v0-(m,0,1)}
We take a connected closed subgroup $L_{(m,0,1)}$ as $Q_{(m,0,1)}$ 
and a submanifold $S_{(m,0,1)}$ as $M_{(m,0,1)}^1$. 
Then, $L_{(m,0,1)}\cdot S_{(m,0,1)}$ coincides with $D_{(m,0,1)}$. 
\end{proposition}

\begin{proof}
Clearly, $Q_{(m,0,1)}\cdot S_{(m,0,1)}$ is contained in $D_{(m,0,1)}$. 
Then, we see the opposite inclusion. 
Let us take an element $v_{(m,0,1)}(\boldsymbol{x}',\boldsymbol{y})$ in $D_{(m,0,1)}$ 
and write $\boldsymbol{x}'=\boldsymbol{x}_R'+\sqrt{-1}\boldsymbol{x}_I'$ for 
$\boldsymbol{x}_R',\boldsymbol{x}_I'\in \mathbb{R}^{n-m}$ 
and $\boldsymbol{y}=\boldsymbol{y}_R+\sqrt{-1}\boldsymbol{y}_I$ for 
$\boldsymbol{y}_R,\boldsymbol{y}_I\in \mathbb{R}^{n}$. 
Then, we have 
\begin{align*}
v_{(m,0,1)}(\boldsymbol{x}',\boldsymbol{y})
=g((\boldsymbol{0}_m,\boldsymbol{x}_R'),\boldsymbol{y}_R,0)\cdot 
	v_{(m,0,1)}(\sqrt{-1}\boldsymbol{x}_I',\sqrt{-1}\boldsymbol{y}_I). 
\end{align*}
This implies $v_{(m,0,1)}(\boldsymbol{x}',\boldsymbol{y})\in L_{(m,0,1)}\cdot S_{(m,0,1)}$. 
Hence, we have verified $D_{(m,0,1)}\subset Q_{(m,0,1)}\cdot S_{(m,0,1)}$. 
Since it is obvious that $Q_{(m,0,1)}\cdot S_{(m,0,1)}\subset D_{(m,0,1)}$, 
we have shown $Q_{(m,0,1)}\cdot S_{(m,0,1)}=D_{(m,0,1)}$. 
\end{proof}

For the special case $m=n$, we can show the following proposition 
by the same way as Proposition \ref{prop:v0-(m,0,1)}. 

\begin{proposition}
\label{prop:v0-(n,0,1)}
We take $L_{(p,n,1)}$ as $Q_{(p,n,1)}$ and 
$S_{(n,0,1)}$ as $M_{(n,0,1)}^1=\{ v_{(n,0,1)}(\sqrt{-1}\boldsymbol{b}):
\boldsymbol{b}\in \mathbb{R}^n\} $. 
Then, the set $L_{(n,0,1)}\cdot S_{(n,0,1)}$ coincides with $D_{(n,0,1)}$. 
\end{proposition}

Therefore, we conclude: 

\begin{theorem}
\label{thm:visible-(m,0,1)}
The $Q_{(m,0,1)}$-action on $D_{(m,0,1)}=G^{\mathbb{C}}/H_{(m,0,1)}^{\mathbb{C}}$ 
is strongly visible with $(2n-m)$-dimensional slice 
$M_{(m,0,1)}^1\simeq \exp \sqrt{-1}\mathfrak{q}_{(m,0,1)}$. 
\end{theorem}


\section{Visible action on $G^{\mathbb{C}}/H_{(m,0,0)}^{\mathbb{C}}$}
\label{sec:(m,0,0)}

This section deals with the case where a complex Heisenberg homogeneous space is 
$D_{(m,0,0)}:=G^{\mathbb{C}}/H_{(m,0,0)}^{\mathbb{C}}$ for $1\leq m\leq n$. 
According to Lemma \ref{lem:q} for $\mathfrak{q}_{(m,0,0)}$, 
we divide into two cases: $1\leq m<n$ and $m=n$. 

Let us consider a general case $1\leq m<n$. 
The complementary subspace $\mathfrak{q}_{(m,0,0)}$ is given by 
\begin{align*}
\mathfrak{q}_{(m,0,0)}=\operatorname{span}_{\mathbb{R}}\{ X_{m+1},\ldots ,X_n,Y_1,\ldots ,Y_n,Z\} , 
\end{align*}
and then 
\begin{align*}
Q_{(m,0,0)}
&=\{ g((\boldsymbol{0}_{m},\boldsymbol{s}'),\boldsymbol{t},u):
	\boldsymbol{s}'\in \mathbb{R}^{n-m},\boldsymbol{t}\in \mathbb{R}^n,u\in \mathbb{R}\} .
\end{align*}
Thus, $Q_{(m,0,0)}$ is isomorphic to $\mathbb{R}^m\times {\bf{H}}_{n-m}$ 
which coincides with $Q_{(m,0,1)}$ (see Section \ref{subsec:action-(m,0,1)}). 
It is noteworthy that a closed subgroup of $G$ which we will consider 
is not $Q_{(m,0,0)}$ but $Q_{(m-1,0,0)}$ which contains $Q_{(m,0,0)}$. 
On the other hand, 
the $Q_{(m,0,0)}$-action on $D_{(m,0,0)}$ will be discussed in Theorem \ref{thm:(m,0,0)-action}, 
separated from this section. 
If $m>1$, then $Q_{(m-1,0,0)}$ forms 
\begin{align*}
Q_{(m-1,0,0)}
=\{ g((\boldsymbol{0}_{m-1},s_m,\boldsymbol{s}'),\boldsymbol{t},u):
	\boldsymbol{s}'\in \mathbb{R}^{n-m},\boldsymbol{t}\in \mathbb{R}^n,s_m,u\in \mathbb{R}\} , 
\end{align*}
whereas, if $m=1$ then $Q_{(0,0,0)}$ coincides with $G$. 
Hence, the main object of this section is the $Q_{(m-1,0,0)}$-action on $D_{(m,0,0)}$. 

For a special case $m=n$, 
the Lie algebra $\mathfrak{q}_{(n,0,0)}$ is an abelian subalgebra of maximal dimension. 
Then, we will consider the $Q_{(n,0,0)}$-action on $D_{(n,0,0)}$. 
Here, 
$Q_{(n,0,0)}=\{ g(\boldsymbol{0},\boldsymbol{t},u):\boldsymbol{t}\in \mathbb{R}^n,u\in \mathbb{R}\} 
\simeq \mathbb{R}^{n+1}$. 

\subsection{$Q_{(m-1,0,0)}$-action on $D_{(m,0,0)}$ for $1\leq m<n$}
\label{subsec:action-(m,0,0)}

First, we consider the case $1\leq m<n$ 
and the $Q_{(m-1,0,0)}$-action on $D_{(m,0,0)}$. 
Now, we write 
\begin{align*}
v_{(m,0,0)}(\boldsymbol{x}',\boldsymbol{y},z)
	:=g((\boldsymbol{0}_{m},\boldsymbol{x}'),\boldsymbol{y},z)H_{(m,0,0)}^{\mathbb{C}} 
\end{align*}
for $\boldsymbol{x}'\in \mathbb{C}^{n-m},\boldsymbol{y}\in \mathbb{C}^n,z\in \mathbb{C}$. 
Then, the complex Heisenberg homogeneous space $D_{(m,0,0)}$ is 
\begin{align*}
D_{(m,0,0)}=\{ v_{(m,0,0)}(\boldsymbol{x}',\boldsymbol{y},z):
	\boldsymbol{x}'\in \mathbb{C}^{n-m},\boldsymbol{y}\in \mathbb{C}^n,z\in \mathbb{C}\}  
\end{align*}

\begin{lemma}
\label{lem:action-(m,0,0)}
The $Q_{(m-1,0,0)}$-action on $D_{(m,0,0)}$ is written by 
\begin{multline*}
g((\boldsymbol{0}_{m-1},s_m,\boldsymbol{s}'),\boldsymbol{t},u)\cdot 
	v_{(m,0,0)}(\boldsymbol{x}',\boldsymbol{y},z)\\
=v_{(m,0,0)}(\boldsymbol{s}'+\boldsymbol{x}',
	\boldsymbol{t}+\boldsymbol{y},u+z+s_m(t_m+y_m)-(\boldsymbol{t}'|\boldsymbol{x}')) 
\end{multline*}
for $g((\boldsymbol{0}_{m-1},s_m,\boldsymbol{s}'),\boldsymbol{t},u)\in Q_{(m-1,0,0)}$ 
and $v_{(m,0,0)}(\boldsymbol{x}',\boldsymbol{y},z)\in D_{(m,0,0)}$. 
Here, for $\boldsymbol{t}=(t_1,\ldots ,t_n)\in \mathbb{C}^n$ 
we write $\boldsymbol{t}'=(t_{m+1},\ldots ,t_m)$. 
\end{lemma}

\begin{proof}
It follows from Lemma \ref{lem:multiple} that 
\begin{multline}
\label{eq:action-(m,0,0)}
g((\boldsymbol{0}_{m-1},s_m,\boldsymbol{s}'),\boldsymbol{t},u)\cdot 
	v_{(m,0,0)}(\boldsymbol{x}',\boldsymbol{y},z)\\
=g((\boldsymbol{0}_{m-1},s_m,\boldsymbol{s}'+\boldsymbol{x}'),
	\boldsymbol{t}+\boldsymbol{y},u+z-(\boldsymbol{t}'|\boldsymbol{x}'))
		H_{(m,0,0)}^{\mathbb{C}}. 
\end{multline}

The element $e^{s_mX_m}$ commutes with $e^{(s_i+x_i)X_i}$ for all $i$, 
with $e^{(t_j+y_j)Y_j}$ for all $j\neq m$ 
and with $e^{(u+z-(\boldsymbol{t}'|\boldsymbol{x}'))Z}$, 
whereas 
it does not coincide with $e^{(t_m+y_m)Y_m}$ 
because Lemma \ref{lem:exp} implies 
\begin{align*}
e^{s_mX_m}e^{(t_m+y_m)Y_m}
&=(e^{s_mX_m}e^{(t_m+y_m)Y_m}e^{-s_mX_m})e^{s_mX_m}\\
&=e^{(t_m+y_m)Y_m}e^{s_m(t_m+y_m)Z}e^{s_mX_m}. 
\end{align*}
Then, we obtain 
\begin{multline*}
g((\boldsymbol{0}_{m-1},s_m,\boldsymbol{s}'+\boldsymbol{x}'),
\boldsymbol{t}+\boldsymbol{y},u+z-(\boldsymbol{t}'|\boldsymbol{x}'))\\
=g((\boldsymbol{0}_{m-1},0,\boldsymbol{s}'+\boldsymbol{x}'),
\boldsymbol{t}+\boldsymbol{y},u+z+s_m(t_m+y_m)-(\boldsymbol{t}'|\boldsymbol{x}'))\cdot e^{s_mX_m}. 
\end{multline*}
As $e^{s_mX_m}\in H_{(m,0,0)}^{\mathbb{C}}$, we obtain 
\begin{align*}
(\ref{eq:action-(m,0,0)})=v_{(m,0,0)}(\boldsymbol{s}'+\boldsymbol{x}',
	\boldsymbol{t}+\boldsymbol{y},u+z+s_m(t_m+y_m)-(\boldsymbol{t}'|\boldsymbol{x}')). 
\end{align*}
Hence, Lemma \ref{lem:action-(m,0,0)} has been proved. 
\end{proof}

\subsection{Verification of (\ref{v0}) for $Q_{(m-1,0,0)}$-action on $D_{(m,0,0)}$}
\label{subsec:v0-(m,0,0)}

In this subsection, we will find a submanifold $S_{(m,0,0)}$ of the real form 
$M_{(m,0,0)}^1=D_{(m,0,0)}^{\sigma _1}$ explicitly 
such that $L_{(m,0,0)}\cdot S_{(m,0,0)}$ is open in $D_{(m,0,0)}$. 
Here, 
$M_{(m,0,0)}^1$ is of the form 
\begin{align*}
M_{(m,0,0)}^1=\{ v_{(m,0,0)}(\sqrt{-1}\boldsymbol{a}',\sqrt{-1}\boldsymbol{b},c):
	\boldsymbol{a}'\in \mathbb{R}^{n-m},\boldsymbol{b}\in \mathbb{R}^n,c\in \mathbb{R}\} . 
\end{align*}

\begin{proposition}
\label{prop:v0-(m,0,0)}
We take a closed subgroup $L_{(m,0,0)}$ as $Q_{(m-1,0,0)}$ and 
a submanifold $S_{(m,0,0)}$ in $M_{(m,0,0)}^1$ as 
\begin{align}
\label{eq:slice-(m,0,0)}
S_{(m,0,0)}
&=\{ v_{(m,0,0)}(\sqrt{-1}\boldsymbol{a}',\sqrt{-1}\boldsymbol{b},0)\in M_{(m,0,0)}^1:
	\boldsymbol{b}\in (\mathbb{R}^{\times })^n\} . 
\end{align}
Then, the set $L_{(m-1,0,0)}\cdot S_{(m,0,0)}$ equals 
\begin{align}
D'_{(m,0,0)}=\{ v_{(m,0,0)}(\boldsymbol{x}',\boldsymbol{y},z)\in D_{(m,0,0)}:
	\boldsymbol{y}_I\in (\mathbb{R}^{\times })^n\} 
\end{align}
and then is an open set in $D_{(m,0,0)}$. 
Here, an element $\boldsymbol{y}\in \mathbb{C}^n$ is written as 
$\boldsymbol{y}=\boldsymbol{y}_R+\sqrt{-1}\boldsymbol{y}_I$ 
for some $\boldsymbol{y}_R,\boldsymbol{y}_I\in \mathbb{R}^n$. 
\end{proposition}

\begin{proof}
Let us take an element $v_{(m,0,0)}(\boldsymbol{x}',\boldsymbol{y},z)\in D'_{(m,0,0)}$. 
We write $\boldsymbol{x}'=\boldsymbol{x}'_R+\sqrt{-1}\boldsymbol{x}'_I$ 
with $\boldsymbol{x}'_R,\boldsymbol{x}'_I\in \mathbb{R}^{n-m}$ 
and $z=z_R+\sqrt{-1}z_I$ with $z_R,z_I\in \mathbb{R}$. 
Further, we denote by $(y_m)_R$ the $m$-th component of $\boldsymbol{y}_R$, 
equivalently, the real part of $y_m$. 

We notice that $(y_m)_I\neq 0$ since $\boldsymbol{y}_I\in (\mathbb{R}^{\times })^n$. 
Here, we set 
\begin{gather*}
\boldsymbol{s}':=\boldsymbol{x}_R',~
\boldsymbol{t}:=\boldsymbol{y}_R,~
\boldsymbol{a}':=\boldsymbol{x}_I',~
\boldsymbol{b}:=\boldsymbol{y}_I,\\
s_m:=(z_I+(\boldsymbol{y}'_R|\boldsymbol{x}_I'))(y_m)_I^{-1},~
u:=z_R-s_m(y_m)_R. 
\end{gather*}
Then, the following equality is fulfilled: 
\begin{align}
\label{eq:v0-(m,0,0)}
v_{(m,0,0)}(\boldsymbol{x}',\boldsymbol{y},z)
=g((\boldsymbol{0}_{m-1},s_m,\boldsymbol{s}'),\boldsymbol{t},u)\cdot 
	v_{(m,0,0)}(\sqrt{-1}\boldsymbol{a}',\sqrt{-1}\boldsymbol{b},0). 
\end{align}
Hence, this implies $D'_{(m,0,0)}\subset Q_{(m-1,0,0)}\cdot S_{(m,0,0)}$. 

Equality (\ref{eq:v0-(m,0,0)}) also shows that $Q_{(m-1,0,0)}\cdot S_{(m,0,0)}\subset D'_{(m,0,0)}$ 
because $\boldsymbol{y}_I=\boldsymbol{b}\in (\mathbb{R}^{\times })^n$. 
Therefore, Proposition \ref{prop:v0-(m,0,0)} has been proved. 
\end{proof}

We take a submanifold $\mathfrak{q}'_{(m,0,0)}$ in $\mathfrak{q}_{(m,0,0)}$ as 
\begin{align}
\label{eq:q'}
\mathfrak{q}'_{(m,0,0)}
=\sum _{i=m+1}^n \mathbb{R}X_i+\sum _{j=1}^n (\mathbb{R}^{\times})Y_j, 
\end{align}
Then, $S_{(m,0,0)}$ is diffeomorphic to $\exp \sqrt{-1}\mathfrak{q}'_{(m,0,0)}$. 
As a consequence, we get: 

\begin{theorem}
\label{thm:visible-(m,0,0)}
For $1\leq m\leq n$, 
the $Q_{(m-1,0,0)}$-action on $D_{(m,0,0)}$ is strongly visible 
with $(2n-m)$-dimensional slice $S_{(m,0,0)}\simeq \exp \sqrt{-1}\mathfrak{q}'_{(m,0,0)}$. 
\end{theorem}

\subsection{$Q_{(n,0,0)}$-action on $D_{(n,0,0)}$}
\label{subsec:(n,0,0)}

In this subsection, we consider the $Q_{(n,0,0)}$-action on $D_{(n,0,0)}$. 
We recall that 
$Q_{(n,0,0)}=\{ g(\boldsymbol{0},\boldsymbol{t},u):\boldsymbol{t}\in \mathbb{R}^n,u\in \mathbb{R}\} $ 
and $D_{(n,0,0)}=\{ v_{(n,0,0)}(\boldsymbol{y},z):=g(\boldsymbol{0},\boldsymbol{y},z)H_{(n,0,0)}^{\mathbb{C}}:
\boldsymbol{y}\in \mathbb{C}^n,~z\in \mathbb{C}\} $. 
Then, this action is given by 
\begin{align*}
g(\boldsymbol{0},\boldsymbol{t},u)\cdot v_{(n,0,0)}(\boldsymbol{y},z)
=v_{(n,0,0)}(\boldsymbol{t}+\boldsymbol{y},u+z)
\end{align*}
for $g(\boldsymbol{0},\boldsymbol{t},u)\in Q_{(n,0,0)}$ and 
$v_{(n,0,0)}(\boldsymbol{y},z)\in D_{(n,0,0)}$. 

In this setting, 
we take an anti-holomorphic diffeomorphism on $D_{(n,0,0)}$ as $\sigma _2$ (see (\ref{eq:sigma})). 
Then, the real form $M_{(n,0,0)}^2=D_{(n,0,0)}^{\sigma _2}$ is given by 
\begin{align}
\label{eq:slice-(n,0,0)}
M_{(n,0,0)}^2=\{ v_{(n,0,0)}(\sqrt{-1}\boldsymbol{b},\sqrt{-1}c):
\boldsymbol{b}\in \mathbb{R}^n,c\in \mathbb{R}\} . 
\end{align}
Then, we have: 

\begin{proposition}
\label{prop:v0-(n,0,0)}
We take $L_{(n,0,0)}$ as $Q_{(n,0,0)}$ and $S_{(n,0,0)}$ as $M_{(n,0,0)}^2$. 
Then, $L_{(n,0,0)}\cdot S_{(n,0,0)}$ coincides with $D_{(n,0,0)}$. 
\end{proposition}

Hence, we obtain: 

\begin{theorem}
\label{thm:visible-(n,0,0)}
The $Q_{(n,0,0)}$-action on $D_{(n,0,0)}$ is strongly visible 
with $(n+1)$-dimensional slice 
$M_{(n,0,0)}^2\simeq \exp \sqrt{-1}\mathfrak{q}_{(n,0,0)}$. 
\end{theorem}


\section{Proof of Theorem \ref{thm:main} and proof of Corollary \ref{cor:main}}
\label{sec:proof}

In this section, 
we will completely give a proof of Theorem \ref{thm:main} and a proof of Corollary \ref{cor:main}. 

\begin{proof}[Proof of Theorem \ref{thm:main}]
Two equivalent holomorphic actions keeps the strong visibility 
(see Proposition \ref{prop:classification}). 
Then, it suffices to show Theorem \ref{thm:main} 
for $D_{(0,0,1)}$, $D_{(p,q,1)}$, $D_{(m,0,1)}$ and $D_{(m,0,0)}$. 
We have already proved the strong visibility for 
the $Q_{(0,0,1)}$-action on $D_{(0,0,1)}$ in Theorem \ref{thm:visible-(0,0,1)}; 
the $Q_{(p,q,1)}$-action on $D_{(p,q,1)}$ in Theorem \ref{thm:visible-(p,q,1)}; 
the $Q_{(m,0,1)}$-action on $D_{(m,0,1)}$ in Theorem \ref{thm:visible-(m,0,1)}; 
the $Q_{(m'-1,0,0)}$-action on $D_{(m',0,0)}$ in Theorem \ref{thm:visible-(m,0,0)} if $1\leq m'<n$; 
and the $Q_{(n,0,0)}$-action on $D_{(n,0,0)}$ in Theorem \ref{thm:visible-(n,0,0)}. 
Therefore, the proof of Theorem \ref{thm:main} has been accomplished. 
\end{proof}

Table \ref{table:choice} indicates our choice of a connected closed subgroup $L_{(k,\ell ,\varepsilon )}$ 
and a slice $S_{(k,\ell ,\varepsilon )}$ for the strongly visible action on $D_{(k,\ell ,\varepsilon )}$ 
for each non-trivial connected closed subgroup $H_{(k,\ell ,\varepsilon )}\subset G$. 
Here, the positive integer $m'$ in the fifth line is restricted to $1\leq m'<n$. 

\begin{table}[htbp]
	\begin{align*}
	\begin{array}{c|cccccc}
	H_{(k,\ell ,\varepsilon )} & 
		L_{(k,\ell ,\varepsilon )}
		& \sigma 
		& \mathfrak{s}_{(k,\ell ,\varepsilon )} 
		& \dim S_{(k,\ell ,\varepsilon )}
		& \mbox{Theorem}\\
	\cline{1-6} 
	H_{(0,0,1)} & Q_{(0,0,1)}
		& \sigma _1
		& \sqrt{-1}\mathfrak{q}_{(0,0,1)}
		& 2n 
		& \ref{thm:visible-(0,0,1)}\\
	H_{(p,q,1)} & Q_{(p,q,1)}
		& \sigma _1
		& \sqrt{-1}\mathfrak{q}_{(p,q,1)} 
		& 2n-p-q
		& \ref{thm:visible-(p,q,1)}\\
	H_{(m,0,1)} & Q_{(m,0,1)}
		& \sigma _1
		& \sqrt{-1}\mathfrak{q}_{(m,0,1)} 
		& 2n-m
		& \ref{thm:visible-(m,0,1)}\\
	H_{(m',0,0)} & Q_{(m'-1,0,0)}
		& \sigma _1
		& \sqrt{-1}\mathfrak{q}'_{(m',0,0)} 
		& 2n-m'
		& \ref{thm:visible-(m,0,0)} \\
	H_{(n,0,0)} & Q_{(n,0,0)}
		& \sigma _2
		& \sqrt{-1}\mathfrak{q}_{(n,0,0)}
		& n+1
		& \ref{thm:visible-(n,0,0)}
	\end{array}
	\end{align*}
	\caption{$L_{(k,\ell ,\varepsilon )}$, $\sigma$ and $S_{(k,\ell ,\varepsilon )}
	\simeq \exp \mathfrak{s}_{(k,\ell ,\varepsilon )}$}
	\label{table:choice}
\end{table}

Next, we will give a proof of Corollary \ref{cor:main}. 
More precisely, 
we will show that 
our choice of $S$ and anti-holomorphic diffeomorphism $\sigma$ for the strongly visible $L$-action on $D$ 
satisfies three conditions (\ref{v0})--(\ref{s2}) for the $G$-action on $D$. 

\begin{proof}[Proof of Corollary \ref{cor:main}]
By our proof of Theorem \ref{thm:main}, 
we take a slice $S$, an anti-holomorphic involution $\widetilde{\sigma }$ 
and the induced anti-holomorphic diffeomorphism $\sigma $ 
satisfying (\ref{v0})--(\ref{s2}) for the strongly visible $L$-action on $D$. 

First, we consider the subset $D':=G\cdot S$ in $D$. 
This is of the form 
\begin{align}
\label{eq:open condition}
D'=G\cdot (L\cdot S)=\bigcup _{g\in G}g\cdot (L\cdot S). 
\end{align}
Since $L\cdot S$ is open in $D$, so is $g\cdot (L\cdot S)$ for any $g\in G$. 
Hence, $D'$ is open in $D$. 

Next, it is clear that $\sigma |_{S}=\operatorname{id}_S$. 
Further, $\sigma $ preserves each $G$-orbit in $D':=G\cdot S$. 
Indeed, we take an element $v\in D'$ and write $v=g\cdot s$ for some $g\in G$ and $s\in S$ 
according to (\ref{eq:open condition}). 
By the same argument of (\ref{eq:s2}), we have $\sigma (v)=\widetilde{\sigma }(g)g^{-1}\cdot v\in G\cdot v$. 

Hence, we have shown Corollary \ref{cor:main}. 
\end{proof}


\section{Application to representation theory}
\label{sec:representation}

This section presents new multiplicity-free theorems for unitary and irreducible 
representations of the Heisenberg group of infinite dimension (hence not characters). 
These theorems are built on the investigation 
of strongly visible actions on complex Heisenberg homogeneous spaces. 

\subsection{Propagation theory of multiplicity-freeness property}
\label{subsec:propagation}

Originally, 
the notion of strongly visible actions has been introduced 
as a geometric point of view of the propagation theory of multiplicity-freeness property 
established by T. Kobayashi (see the original papers \cite{ko05,ko13}). 
In this subsection, 
we explain a multiplicity-free theorem of the Heisenberg group 
which is gained as an application of our results on visible actions on complex Heisenberg homogeneous spaces. 

Before that, let us give a quick review on the propagation theory of multiplicity-freeness property, 
based on \cite{ko13}. 
Let $L$ be a Lie group and $\mathcal{V}\to D$ a $L$-equivariant Hermitian holomorphic vector bundle 
over a connected complex manifold $D$. 
We denote by $\mathcal{O}(D,\mathcal{V})$ the space of holomorphic sections of $\mathcal{V}\to D$. 
This carries a Fr\'echet topology by the uniform convergence on compact sets. 
Then, we naturally define a continuous representation $\varpi$ of $L$ 
on the space $\mathcal{O}(D,\mathcal{V})$ of holomorphic sections by 
\begin{align*}
[\varpi (g)s](x):=g\cdot s(g^{-1}\cdot x)
\end{align*}
for $g\in L$, $s\in \mathcal{O}(D,\mathcal{V})$ and $x\in D$. 
Under the setting, the propagation theory of multiplicity-freeness property is explained as follows. 
Here, we say that 
an anti-holomorphic diffeomorphism $\sigma$ on $D$ lifts 
to an anti-holomorphic bundle endomorphism on $\mathcal{V}$ 
if $\sigma $ gives rise to an anti-holomorphic diffeomorphism on $\mathcal{V}$, 
denoted by the same letter $\sigma$, 
such that the restriction of $\sigma$ to the fiber $\mathcal{V}_x$ at each $x\in D$ 
is an isometric and anti-linear map to the fiber $\mathcal{V}_{\sigma (x)}$ at $\sigma (x)$. 

\begin{fact}[{\cite[Theorem 4.3]{ko13}}]
\label{fact:propagation}
Let $\mathcal{V}\to D$ be an $L$-equivariant Hermitian holomorphic vector bundle. 
Suppose that the following conditions are satisfied: 
\begin{enumerate}
	\renewcommand{\theenumi}{\alph{enumi}}
	\item 
	\label{item:propagation-a}
	The $L$-action on $D$ is strongly visible. 
	In particular, 
	there exist a real submanifold $S$ in $D$ and an anti-holomorphic diffeomorphism $\sigma$ 
	on $D':=L\cdot S$ satisfying the conditions (\ref{v0})--(\ref{s2}) 
	equipped with an anti-holomorphic automorphism $\widetilde{\sigma }$ of $L$ 
	satisfying $\sigma (g\cdot x)=\widetilde{\sigma }(g)\cdot \sigma (x)~(g\in L,x\in D)$. 
	\item 
	\label{item:propagation-b}
	For each $x\in S$, the unitary representation of the isotropy subgroup $L_x$ on 
	the fiber $\mathcal{V}_x$ is multiplicity-free. 
	\item 
	\label{item:propagation-c}
	Under the condition (\ref{item:propagation-b}), 
	we write $\mathcal{V}_x=\oplus _{i=1}^{m(x)}\mathcal{V}_x^{(i)}$ 
	for the irreducible decomposition as a representation of $L_x$. 
	Then, $\sigma$ lifts to an anti-holomorphic bundle endomorphism on $\mathcal{V}$ 
	such that 
	$\sigma (\mathcal{V}_x^{(i)})=\mathcal{V}_{x}^{(i)}$ for each $i=1,2,\ldots ,n(x)$ and each $x\in S$. 
\end{enumerate}

Then, any unitary representation which is realized in $\mathcal{O}(D,\mathcal{V})$ is multiplicity-free. 
In particular, 
the continuous representation $\varpi$ of $L$ on $\mathcal{O}(D,\mathcal{V})$ is multiplicity-free. 
\end{fact}

This theory also plays an important role that 
one can give an unified explanation of the multiplicity-freeness for various kinds of representations. 
On the other hand, 
once we give an example of strongly visible actions, 
we expect that one can find several multiplicity-free representations. 

So far, we have found the strongly visible actions on the complex Heisenberg homogeneous spaces 
(see Theorem \ref{thm:main}). 
Applying Fact \ref{fact:propagation} to our setting, 
we can explain the multiplicity-freeness of a kind of representations of the Heisenberg group as follows. 

Let $G$ be the  Heisenberg group 
and $H$ a non-trivial connected closed subgroup of $G$. 
Due to Theorem \ref{thm:main}, 
one can find a connected closed subgroup $L$ of $G$ such that 
the $L$-action on the complex Heisenberg homogeneous space $D=G^{\mathbb{C}}/H^{\mathbb{C}}$ 
is strongly visible. 
Further, it follows from Corollary \ref{cor:main} that 
the $G$-action on $D$ is strongly visible. 

Let $\mathcal{V}$ be a trivial line bundle $D\times \mathbb{C}$ over $D$. 
Then, $\mathcal{O}(D,D\times \mathbb{C})$ is naturally identified with 
the space $\mathcal{O}(D)$ of holomorphic functions on $D$ 
via the map $\mathcal{O}(D,D\times \mathbb{C})\to \mathcal{O}(D)$, $s\mapsto f$ 
which is characterized by $s(x)=(x,f(x))$ ($x\in D$). 
Now, we let $G$ act holomorphically on $\mathcal{V}=D\times \mathbb{C}$ by $g\cdot (x,z):=(g\cdot x,z)$ 
($g\in G$, $x\in D$ and $z\in \mathbb{C}$). 
Then, $\mathcal{V}\to D$ is $G$-equivariant 
and 
\begin{align*}
[\varpi (g)s](x)=(x,f(g^{-1}\cdot x))
\end{align*}
for $s\in \mathcal{O}(\mathcal{V},D)$, $g\in G$ and $x\in D$. 
Hence, we define a representation $\pi$ of $G$ on $\mathcal{O}(D)$ by 
\begin{align}
\label{eq:pi}
[\pi (g)f](x):=f(g^{-1}\cdot x) 
\end{align}
for $g\in G$, $f\in \mathcal{O}(D)$ and $x\in D$, 
and then $\varpi $ is equivalent to $\pi$. 

\begin{theorem}
\label{thm:multiplicity-free}
Let $G$ be the Heisenberg group 
and $H$ a non-trivial connected closed subgroup of $G$. 
Then, the continuous representation $(\pi ,\mathcal{O}(D))$ of $G$ is multiplicity-free. 
Moreover, 
let $L$ be a connected closed subgroup of $G$ such that 
the $L$-action on $G^{\mathbb{C}}/H^{\mathbb{C}}$ is strongly visible. 
Then, 
the restriction $\pi |_L$ of $\pi$ to $L$ is still multiplicity-free. 
\end{theorem}

\begin{proof}
By the proof of Theorem \ref{thm:main}, 
one can find a real submanifold $S$, an anti-holomorphic involution $\widetilde{\sigma }$ on $G^{\mathbb{C}}$ 
and an anti-holomorphic diffeomorphism $\sigma $ on $G^{\mathbb{C}}/H^{\mathbb{C}}$ such that 
(\ref{v0})--(\ref{s2}) are satisfied, $\widetilde{\sigma }$ stabilizes $L$ 
and $\sigma (g\cdot x)=\widetilde{\sigma }(g)\cdot \sigma (x)$ 
for $g\in L$, $x\in G^{\mathbb{C}}/H^{\mathbb{C}}$ (see Remark \ref{rem:compatible}). 
Thus, condition (\ref{item:propagation-a}) of Fact \ref{fact:propagation} has been verified. 

Let $\mathcal{V}$ be a trivial bundle $D\times \mathbb{C}$ over $D$. 
Since each fiber $\mathcal{V}_x\simeq \mathbb{C}$ ($x\in D$) is one-dimensional, 
the unitary representation of $L_x$ on $\mathcal{V}_x$ is irreducible for any $x\in S$, 
in particular, this is multiplicity-free. 
Then, condition (\ref{item:propagation-b}) is satisfied. 

The anti-holomorphic diffeomorphism $\sigma$ on $D$ lifts to an anti-holomorphic bundle endomorphism 
on $\mathcal{V}$ as follows: 
\begin{align*}
\sigma (x,z)=(\sigma (x),\overline{z})\quad (x\in D,z\in \mathbb{C}). 
\end{align*}
Clearly, $\sigma (\mathcal{V}_x)=\mathcal{V}_x$ for all $x\in S$, 
from which condition (\ref{item:propagation-c}) holds. 

Hence, it follows from Fact \ref{fact:propagation} that 
the continuous representation $(\varpi ,\mathcal{O}(D,\mathcal{V}))$ of $L$ is multiplicity-free. 
Therefore, $\pi$ is multiplicity-free as a representation of $L$ since $\varpi \simeq \pi $. 
\end{proof}


\subsection{Geometric criterion of multiplicity-freeness for associated quasi-regular representation}
\label{subsec:criterion}

This section investigates the relationship between strongly visible actions on complex Heisenberg 
homogeneous spaces and the multiplicity-free quasi-regular representations of the Heisenberg group. 

First of all, 
we will mention a geometric criterion of multiplicity-freeness 
of the quasi-regular representation $\pi_{H}$ in Theorem \ref{thm:criterion}, 
and its proof will be given in Sections \ref{subsec:mf-h} and \ref{subsec:visible-h}. 

Let $G=\exp \mathfrak{g}$ be the Heisenberg group 
and $H=\exp \mathfrak{h}$ a non-trivial connected closed subgroup of $G$. 
We set $d:=\dim G-\dim H$. 
Then, there exists a co-exponential basis 
$\{ W_1,\ldots ,W_{d}\} $ to $\mathfrak{h}$ in $\mathfrak{g}$, 
which means that the map $H\times \mathbb{R}^{d}\to G$, 
$(h,w_1,\ldots ,w_{d})\mapsto he^{w_1W_1}\cdots e^{w_{d}W_{d}}$ is a diffeomorphism. 
We set $\mathfrak{q}:=\operatorname{span}_{\mathbb{R}}\{ W_1,\ldots ,W_{d}\} $
and $Q:=\langle \exp \mathfrak{q}\rangle $. 

Let $d\mu $ be a $G$-invariant measure on the  homogeneous space $G/H$ 
which is induced from the Lebesgue measure on $\mathfrak{q}$ via the diffeomorphism 
$\mathfrak{q}\simeq \exp \mathfrak{q}\simeq G/H$. 
We note that $G$-invariant measures on $G/H$ are unique up to scalars. 
We denote by $L^2(G/H)$ the space of square integrable functions on $G/H$ with respect to $d\mu $. 
Then, we define the quasi-regular representation $\pi _H$ of $G$ on $L^2(G/H)$ by 
\begin{align}
\label{eq:quasi-regular}
[\pi _H(g)f](xH)=f(g^{-1}xH)\quad (x,g\in G,f\in L^2(G/H))
\end{align}

The next theorem is a new characterization for the unitary representation 
$\pi _H$ of $G$ to be multiplicity-free by the strongly visible action 
on the complex Heisenberg homogeneous space. 

\begin{theorem}
\label{thm:criterion}
Let $G$ be the  Heisenberg  group. 
For a non-trivial connected closed subgroup $H$ of $G$, 
the following conditions are equivalent: 
\begin{enumerate}
	\renewcommand{\theenumi}{\roman{enumi}}
	\item \label{item:mf}
	The quasi-regular representation $\pi _H$ of $G$ is multiplicity-free. 
	\item \label{item:visible-q}
	The $Q$-action on $G^{\mathbb{C}}/H^{\mathbb{C}}$ is strongly visible. 
	\item \label{item:condition-h}
	$H$ is not isomorphic to $H_{(m,0,0)}$ for any $m=1,2,\ldots ,n-1$. 
\end{enumerate}
\end{theorem}

Our proof of Theorem \ref{thm:criterion} consists of two parts. 
The first part is 
to show (\ref{item:mf}) $\Leftrightarrow$ (\ref{item:condition-h}) (see Section \ref{subsec:mf-h}). 
For this, 
we describe the irreducible decomposition of the quasi-regular representation of $G$ 
in terms of the coadjoint orbits and calculate the multiplicities by applying Corwin-Greenleaf formula 
\cite{corwin-greenleaf}. 
The second one is 
to show the equivalence (\ref{item:visible-q}) $\Leftrightarrow$ (\ref{item:condition-h}) 
(see Section \ref{subsec:visible-h}). 
As we have seen in Sections \ref{sec:(0,0,1)}--\ref{sec:proof}, 
the implication (\ref{item:condition-h}) $\Rightarrow $ (\ref{item:visible-q}) is true. 
For the proof of the opposite implication, 
it suffices to show that 
the $Q_{(m,0,0)}$-action on $G^{\mathbb{C}}/H_{(m,0,0)}^{\mathbb{C}}$ is not strongly visible. 
To carry out, we use the propagation theory of multiplicity-freeness property. 

\subsection{Equivalence between (\ref{item:mf}) and (\ref{item:condition-h})}
\label{subsec:mf-h}

This subsection considers the irreducible decomposition 
of the quasi-regular representation $\pi _H$ of $G$. 
It is known that $\pi _H$ can be realized as the representation $\operatorname{Ind}_H^G\boldsymbol{1}$ 
induced from the trivial representation $\boldsymbol{1}$ of $H$. 
In general, 
the irreducible decomposition of 
$\operatorname{Ind}_H^G \chi $ induced from a unitary character $\chi $ of $H$ 
can be formulated in terms of the coadjoint orbits, 
which is known as the Corwin--Greenleaf formula \cite{corwin-greenleaf} 
(see also \cite[Theorem 1.1]{lipsman}). 
Then, we shall apply the Corwin--Greenleaf formula to our setting for the proof of the equivalence 
(\ref{item:mf}) $\Leftrightarrow$ (\ref{item:condition-h}). 

First, we review the basic facts on the coadjoint orbits of the Heisenberg group $G$ of dimension $2n+1$. 
We denote by $\mathfrak{g}^*$ the dual space of $\mathfrak{g}$. 
Let $G$ act on $\mathfrak{g}^*$ by the coadjoint representation, 
namely, $g\cdot \xi :=\operatorname{Ad}^*(g)\xi $ for $g\in G$ and $\xi \in \mathfrak{g}^*$ 
is given by $(g\cdot \xi )(X):=\xi (\operatorname{Ad}(g^{-1})X)$ ($X\in \mathfrak{g}$). 
Let $\mathcal{B}$ be a basis of $\mathfrak{g}$ given by (\ref{eq:basis}) and 
\begin{align}
\label{eq:dual basis}
\mathcal{B}^*:=\{ X_1^*,\ldots ,X_n^*,Y_1^*,\ldots ,Y_n^*,Z^*\} 
\end{align}
the dual basis of $\mathcal{B}$. 
Any element $\xi \in \mathfrak{g}^*$ is written down as 
\begin{align*}
\xi \equiv \xi (\boldsymbol{\alpha },\boldsymbol{\beta },\gamma )
=\sum _{i=1}^n \alpha _iX_i^*+\sum _{j=1}^n \beta _jY_j^*+\gamma Z
\end{align*}
for $\boldsymbol{\alpha }=(\alpha _1,\ldots ,\alpha _n),\boldsymbol{\beta }=(\beta _1,\ldots ,\beta _n)
\in \mathbb{R}^n$ and $\gamma \in \mathbb{R}$. 
A direct computation shows that 
\begin{align*}
g(\boldsymbol{x},\boldsymbol{y},z)\cdot \xi (\boldsymbol{\alpha },\boldsymbol{\beta },\gamma )
=\xi (\boldsymbol{\alpha }+\gamma \boldsymbol{y},\boldsymbol{\beta }-\gamma \boldsymbol{x},\gamma ) 
\end{align*}
for $g(\boldsymbol{x},\boldsymbol{y},z)\in G$ and 
$\xi (\boldsymbol{\alpha },\boldsymbol{\beta },\gamma )\in \mathfrak{g}^*$. 
Then, the coadjoint orbits of the Heisenberg group $G$ are described explicitly as follows. 

\begin{lemma}
\label{lem:coadjoint orbit}
Let $G$ be the Heisenberg group. 
For an element $\xi (\boldsymbol{\alpha },\boldsymbol{\beta },\gamma )\in \mathfrak{g}^*$, we have: 
\begin{enumerate}
	\item If 
	$\gamma \neq 0$, 
	then the coadjoint orbit $G\cdot \xi (\boldsymbol{\alpha },\boldsymbol{\beta },\gamma )$ 
	(of dimension  equals $2n$) is given by:
	\begin{align*}
	G\cdot \xi (\boldsymbol{\alpha },\boldsymbol{\beta },\gamma )
	=\{ \xi (\boldsymbol{x},\boldsymbol{y},\gamma ):\boldsymbol{x},\boldsymbol{y}\in \mathbb{R}^n\} . 
	\end{align*}
	\item 
	The coadjoint orbit $G\cdot \xi (\boldsymbol{\alpha },\boldsymbol{\beta },0)$ equals 
	$\{ \xi (\boldsymbol{\alpha },\boldsymbol{\beta },0)\} $. 
\end{enumerate}
Hence, the subset 
\begin{align}
\label{eq:r}
R:=\{ \xi (\boldsymbol{0},\boldsymbol{0},\gamma ):\gamma \in \mathbb{R}^{\times }\} 
\sqcup \{ \xi (\boldsymbol{\alpha },\boldsymbol{\beta },0)
:\boldsymbol{\alpha },\boldsymbol{\beta }\in \mathbb{R}^n\} \subset \mathfrak{g}^*
\end{align}
turns out to be a cross-section  of the orbit space $\mathfrak{g}^*/G$. 
\end{lemma}

Next, the Kirillov orbit method explains that 
there is a one-to-one correspondence between the unitary dual $\widehat{G}$ and the orbit space 
$\mathfrak{g}^*/G$. 
By Lemma \ref{lem:coadjoint orbit}, 
we write a bijection $R\to \widehat{G}$ as 
$\xi (\boldsymbol{0},\boldsymbol{0},\gamma )\mapsto \tau _{\xi (\boldsymbol{0},\boldsymbol{0},\gamma )}
\equiv \tau _{\gamma }$ 
and $\xi (\boldsymbol{a},\boldsymbol{b},0)\mapsto \tau _{\xi (\boldsymbol{a},\boldsymbol{b},0)}
\equiv \tau _{\boldsymbol{a},\boldsymbol{b}}$. 

We are ready to mention the Corwin--Greenleaf formula for the Heisenberg group $G$. 
The following is refereed by the statement of \cite[Theorem 1.1]{lipsman}. 

\begin{fact}
\label{fact:corwin-greenleaf}
The quasi-regular representation $\pi _H=\operatorname{Ind}_H^G \boldsymbol{1}$ is decomposed 
into the direct integral of irreducible representations of $G$ as 
\begin{align}
\label{eq:irreducible decomposition}
\pi _H=\operatorname{Ind}_H^G \boldsymbol{1}\simeq 
\int _{(G\cdot \mathfrak{q}^*)/G}^{\oplus }m_{\pi _H}(\xi )\tau _{\xi }\,d\xi . 
\end{align}
Here, $d\xi $ is a $G$-invariant measure on $(G\cdot \mathfrak{q}^*)/G$ induced from 
the Lebesgue measure on $\mathfrak{q}^*\subset \mathfrak{g}^*$ via 
$\mathfrak{q}^*\to (G\cdot \mathfrak{q}^*)/G$, 
and $m_{\pi _H}:(G\cdot \mathfrak{q}^*)/G\to \mathbb{N}\cup \{ \infty \}$ is the multiplicity function. 

The multiplicity function $m_{\pi _H}(\xi )$ is either finite and bounded or infinite, 
according to the fact that 
either $2\dim (H\cdot \xi)=\dim (G\cdot \xi )$ or $2\dim (H\cdot \xi )<\dim (G\cdot \xi )$ hold 
for almost all $\xi$ with respect to the measure $d\xi$. 
In the finite case, namely, $2\dim (H\cdot \xi)=\dim (G\cdot \xi )$ for almost all $\xi$, 
$m_{\pi _H}(\xi )$ equals the number of $H$-orbits in $\mathfrak{q}^*\cap (G\cdot \xi )$. 
\end{fact}

\begin{remark}
\label{remark:multiplicity}
According to these two cases, 
we say that $\pi _H$ is of finite or of infinite multiplicities. 
\end{remark}

Our proof of the equivalence (\ref{item:mf}) $\Leftrightarrow$ (\ref{item:condition-h}) 
is carried out by showing $m_{\pi _H}(\xi )=1$ for any generic $\xi \in \mathfrak{q}^*$ 
(namely, the orbit $G\cdot \xi$ is of maximal dimension) 
in (\ref{eq:irreducible decomposition}) for our setting according to Fact \ref{fact:corwin-greenleaf}. 
Before that, we explain the following lemma: 

\begin{lemma}
\label{lem:mf-equiv}
Let $G$ be the Heisenberg group, 
$H_1,H_2$ connected closed subgroups of $G$ 
and $\pi _{H_1},\pi_{H_2}$ the quasi-regular representations of $G$. 
Suppose that $H_1$ is isomorphic to $H_2$. 
Then, $\pi _{H_1}$ is multiplicity-free if and only if $\pi _{H_2}$ also is. 
\end{lemma}

\begin{proof}
By Lemma \ref{lem:automorphism}, 
there exists a Lie group automorphism $\varphi $ on $G=\exp \mathfrak{g}$ such that 
the restriction $\varphi |_{H_1}$ gives rise to the isomorphism from $H_1$ to $H_2$. 
Then, the induced map $\Phi :G/H_1\to G/H_2$, $xH_1\mapsto \varphi (x)H_2$ is a $C^{\infty}$-diffeomorphism. 
Let $d\mu _2$ be a $G$-invariant measure on $G/H_2$ 
and $d\mu _1:=\Phi ^*(d\mu _2)$ the pull-back of $d\mu _2$. 
Then, $d\mu _1$ is a $G$-invariant measure on $G/H_1$. 
Here, if $f\in L^2(G/H_2,d\mu _2)$ then $f\circ \Phi \in L^2(G/H_1,d\mu _1)$, 
which gives rise to the unitary isomorphism $\Phi ^{\vee}:L^2(G/H_2,d\mu _2)\to L^2(G/H_1,d\mu _1)$, 
$f\mapsto \Phi ^{\vee}(f):=f\circ \Phi $. 
Then, we have 
\begin{align*}
\pi _{H_1}(g) \circ \Phi ^{\vee }=\Phi ^{\vee }\circ \pi _{H_2}(\varphi (g))\quad (\forall g\in G). 
\end{align*}
Hence, $\pi _{H_1}$ is equivalent to $\pi _{H_2}\circ \varphi $, which is enough to conclude. 
\end{proof}

Now, we are going to prove the equivalence (\ref{item:mf}) $\Leftrightarrow$ (\ref{item:condition-h}). 

\begin{proof}[Proof of (\ref{item:mf}) $\Leftrightarrow$ (\ref{item:condition-h}) 
in Theorem \ref{thm:criterion}]
Due to Lemma \ref{lem:mf-equiv}, 
it suffices to verify the equivalence 
when $H\equiv H_{(k,\ell ,\varepsilon )}$ are $H_{(0,0,1)}$, $H_{(p,q,1)}$, $H_{(m,0,1)}$ and $H_{(m,0,0)}$. 
We divide our proof into two cases. 

\begin{case}
\label{case:normal}
Let $H_{(k,\ell ,\varepsilon )}$ be either $H_{(0,0,1)},H_{(p,q,1)}$ or $H_{(m,0,1)}$. 
Then, the Lie algebra $\mathfrak{h}_{(k,\ell ,\varepsilon )}$ contains 
the center $\mathfrak{z}(\mathfrak{g})=\mathbb{R}Z$, 
from which $\mathfrak{q}_{(k,\ell ,\varepsilon )}^*$ is a subspace of 
$\operatorname{span}_{\mathbb{R}}\{ X_i^*,Y_j^*:1\leq i,j\leq n\} $. 
By Lemma \ref{lem:coadjoint orbit}, 
we have $G\cdot \xi =\{ \xi \} $ for any $\xi \in \mathfrak{q}_{(k,\ell ,\varepsilon )}^*$. 
Thus, we obtain 
\begin{align}
\label{eq:orbit space-normal}
(G\cdot \mathfrak{q}_{(k,\ell ,\varepsilon )}^*)/G
=\mathfrak{q}_{(k,\ell ,\varepsilon )}^*/G
\simeq \mathfrak{q}_{(k,\ell ,\varepsilon )}^*. 
\end{align}

Clearly, $H\cdot \xi =\{ \xi \} $. 
Hence, we have verified the equality $2\dim (H\cdot \xi )=\dim (G\cdot \xi )$. 
By Fact \ref{fact:corwin-greenleaf}, 
the multiplicity $m_{\pi _H}(\xi )$ is finite for any $\xi \in \mathfrak{q}_{(k,\ell ,\varepsilon )}^*$. 
Further, 
we have $\mathfrak{q}_{(k,\ell ,\varepsilon )}^*\cap (G\cdot \xi )=\{ \xi \} =H\cdot \xi$. 
Hence, it follows from Fact \ref{fact:corwin-greenleaf} that we conclude $m_{\pi _H}(\xi )=1$ 
for any $\xi \in \mathfrak{q}_{(k,\ell ,\varepsilon )}^*$. 
Therefore, $\pi _{H_{(k,\ell ,\varepsilon )}}$ is multiplicity-free. 
\end{case}

\begin{case}
\label{case:(m,0,0)}
Let us consider $H_{(m,0,0)}$ for $1\leq m\leq n$. 
The complementary subspace $\mathfrak{q}_{(m,0,0)}^*$ contains the center $\mathbb{R}Z^*$. 
Thus, 
$\xi (\boldsymbol{\alpha },\boldsymbol{\beta },\gamma )\in \mathfrak{q}_{(m,0,0)}^*$ 
is generic if and only if $\gamma \neq 0$. 

Fix now $\xi (\boldsymbol{\alpha },\boldsymbol{\beta },\gamma )\in \mathfrak{q}_{(m,0,0)}^*$ 
with $\gamma \neq 0$. 
Then, it follows from Lemma \ref{lem:coadjoint orbit} that 
$\dim (G\cdot \xi (\boldsymbol{\alpha },\boldsymbol{\beta },\gamma ))=2n$. 
On the other hand, 
the coadjoint $H_{(m,0,0)}$-orbit through $\xi (\boldsymbol{\alpha },\boldsymbol{\beta },\gamma )$ 
is given as follows: 
If $m<n$ then 
\begin{align*}
H_{(m,0,0)}\cdot \xi (\boldsymbol{\alpha },\boldsymbol{\beta },\gamma )
=\{ \xi (\boldsymbol{\alpha },(\boldsymbol{y}',\boldsymbol{\beta }''),\gamma ):
	\boldsymbol{y}'\in \mathbb{R}^m\} 
\end{align*}
where we shall write $\boldsymbol{\beta }=(\boldsymbol{\beta }',\boldsymbol{\beta }'')\in \mathbb{R}^n$ 
for some $\boldsymbol{\beta }'\in \mathbb{R}^m$ and $\boldsymbol{\beta }''\in \mathbb{R}^{n-m}$ 
and $(\boldsymbol{y}',\boldsymbol{\beta }'')\in \mathbb{R}^n$; 
and if $m=n$ then 
\begin{align}
\label{eq:h-orbit}
H_{(n,0,0)}\cdot \xi (\boldsymbol{\alpha },\boldsymbol{\beta },\gamma )
=\{ \xi (\boldsymbol{\alpha },\boldsymbol{y},\gamma ):\boldsymbol{y}\in \mathbb{R}^n\} . 
\end{align}
These assert that 
\begin{align*}
\dim (H_{(m,0,0)}\cdot \xi (\boldsymbol{\alpha },\boldsymbol{\beta },\gamma ))=m 
\quad (1\leq m\leq n). 
\end{align*}
Hence, the equality $2\dim (H_{(m,0,0)}\cdot \xi (\boldsymbol{\alpha },\boldsymbol{\beta },\gamma ))
=\dim (G\cdot \xi (\boldsymbol{\alpha },\boldsymbol{\beta },\gamma ))$ holds 
if and only if $m=n$. 
It follows from Fact \ref{fact:corwin-greenleaf} 
that the quasi-regular representation $\pi _{(m,0,0)}\equiv \pi _{H_{(m,0,0)}}$ 
is not multiplicity-free if $m<n$, 
in particular, $m_{\pi _{(m,0,0)}}(\xi )=\infty $ for any generic element $\xi \in \mathfrak{q}_{(m,0,0)}^*$. 

Finally, let us show that $\pi _{(n,0,0)}$ is multiplicity-free. 
We recall 
$\mathfrak{q}_{(n,0,0)}^*
=\{ \xi (\boldsymbol{0},\boldsymbol{\beta }',\gamma '):
\boldsymbol{\beta }'\in \mathbb{R}^n,\gamma '\in \mathbb{R}\} $ 
and $G\cdot \xi (\boldsymbol{\alpha },\boldsymbol{\beta },\gamma )
=\{ \xi (\boldsymbol{x},\boldsymbol{y},\gamma ):\boldsymbol{x},\boldsymbol{y}\in \mathbb{R}^n\} $ 
for generic $\xi (\boldsymbol{\alpha },\boldsymbol{\beta },\gamma )\in \mathfrak{q}_{(n,0,0)}^*$. 
Then, we have 
\begin{align}
\label{eq:q-g-orbit}
\mathfrak{q}^*_{(n,0,0)}\cap (G\cdot \xi (\boldsymbol{\alpha },\boldsymbol{\beta },\gamma ))
&=\{ \xi (\boldsymbol{0},\boldsymbol{y},\gamma ):\boldsymbol{y}\in \mathbb{R}^n\} . 
\end{align}
In view of the description (\ref{eq:h-orbit}) of the $H_{(n,0,0)}$-orbit, 
the right-hand side of (\ref{eq:q-g-orbit}) coincides with 
$H_{(m,0,0)}\cdot \xi (\boldsymbol{0},\boldsymbol{0},\gamma ) $. 
Thus, 
$\mathfrak{q}^*_{(n,0,0)}\cap (G\cdot \xi (\boldsymbol{\alpha },\boldsymbol{\beta },\gamma ))$ itself 
is a $H_{(n,0,0)}$-orbit. 
By Fact \ref{fact:corwin-greenleaf}, 
$m_{\pi _{(n,0,0)}}(\xi (\boldsymbol{\alpha },\boldsymbol{\beta },\gamma ))=1$ 
for any generic element 
$\xi (\boldsymbol{\alpha },\boldsymbol{\beta },\gamma )\in \mathfrak{q}_{(n,0,0)}^*$. 
Therefore, $\pi _{(n,0,0)}$ is multiplicity-free. 
\end{case}

As a consequence, the equivalence (\ref{item:mf}) $\Leftrightarrow$ (\ref{item:condition-h}) 
has been completely proved. 
\end{proof}

\subsection{Realization of $(\pi _{(m,0,0)},L^2(G/H_{(m,0,0)}))$ in $(\pi ,\mathcal{O}(D_{(m,0,0)}))$}
\label{subsec:heat kernel}

This subsection provides a realization of the quasi-regular representation 
$\pi _{(m,0,0)}\equiv \pi _{H_{(m,0,0)}}$ on $L^2(G/H_{(m,0,0)})$ of $Q_{(m,0,0)}$ 
through the continuous representation $\pi$ on $\mathcal{O}(D_{(m,0,0)})$ 
in the setting where $1\leq m<n$. 
Namely, 
we find a continuous and injective operator from $L^2(G/H_{(m,0,0)})$ to $\mathcal{O}(D_{(m,0,0)})$ 
which intertwines $\pi _{(m,0,0)}$ and $\pi$. 
The aim of this subsection is Theorem \ref{thm:B}. 

Recall from Section \ref{sec:(m,0,0)} 
that the connected closed subgroup $Q_{(m,0,0)}\subset G$ is given by 
$Q_{(m,0,0)}=\exp \mathfrak{q}_{(m,0,0)}
=\{ g((\boldsymbol{0}_m,\boldsymbol{s}'),\boldsymbol{t},u):\boldsymbol{s}'\in \mathbb{R}^{n-m},~
\boldsymbol{t}\in \mathbb{R}^n,~u\in \mathbb{R}\} $. 
This is isomorphic to $\mathbb{R}^m\times {\bf H}_{n-m}$ as a Lie group through 
the Lie group isomorphism 
\begin{align*}
\mathbb{R}^m\times {\bf H}_{n-m}\stackrel{\sim}{\to }Q_{(m,0,0)},\quad 
(\boldsymbol{y}'',g(\boldsymbol{x}',\boldsymbol{y}',z))\mapsto 
g((\boldsymbol{0}_m,\boldsymbol{x}'),(\boldsymbol{y}'',\boldsymbol{y}),z). 
\end{align*}
Then, let $\mathbb{R}^m\times {\bf H}_{n-m}$ act on $D_{(m,0,0)}=G^{\mathbb{C}}/H_{(m,0,0)}^{\mathbb{C}}$ 
holomorphically via this isomorphism. 
Similarly, the complexification $Q_{(m,0,0)}^{\mathbb{C}}=\exp \mathfrak{q}_{(m,0,0)}^{\mathbb{C}}$ 
is isomorphic to $\mathbb{C}^m\times {\bf H}_{n-m}^{\mathbb{C}}$ as a complex Lie group. 
Combining this isomorphism with Lemma \ref{lem:homogeneous}, we get a biholomorphic map 
\begin{align*}
\begin{array}{c@{~}c@{~}c@{~}c@{~}c}
\Psi &:&\mathbb{C}^m\times {\bf H}_{n-m}^{\mathbb{C}}&\stackrel{\sim}{\to }&D_{(m,0,0)}\\
&& \rotatebox[origin=c]{90}{$\in $} && \rotatebox[origin=c]{90}{$\in $}\\
&&(\boldsymbol{y}'',g(\boldsymbol{x}',\boldsymbol{y}',z))& 
\mapsto &g((\boldsymbol{0}_m,\boldsymbol{x}'),(\boldsymbol{y}'',\boldsymbol{y}),z)H_{(m,0,0)}^{\mathbb{C}} 
\end{array}
. 
\end{align*}
In particular, 
the holomorphic $(\mathbb{R}^m\times {\bf H}_{n-m})$-action on $\mathbb{C}^m\times {\bf H}_{n-m}^{\mathbb{C}}$ 
is equivalent to that on $D_{(m,0,0)}$ via $\Psi$ 
(see Definition \ref{def:equivalent}). 

We also define a continuous representation $\rho$ of $\mathbb{R}^m\times {\bf H}_{n-m}$ 
on the space $\mathcal{O}(\mathbb{C}^m\times {\bf H}_{n-m}^{\mathbb{C}})$ of holomorphic functions 
on $\mathbb{C}^m\times {\bf H}_{n-m}^{\mathbb{C}}$ by 
\begin{align}
\label{eq:representation}
[\rho (\boldsymbol{y},g)f](\boldsymbol{v},x):=f(\boldsymbol{v}-\boldsymbol{y},g^{-1}x) 
\end{align}
for $(\boldsymbol{y},g)\in \mathbb{R}^m\times {\bf H}_{n-m}$, 
$f\in \mathcal{O}(\mathbb{R}^m\times {\bf H}_{n-m})$ 
and $(\boldsymbol{v},x)\in \mathbb{C}^m\times {\bf H}_{n-m}^{\mathbb{C}}$. 
Then, $(\rho ,\mathcal{O}(\mathbb{C}^m\times {\bf H}_{n-m}^{\mathbb{C}}))$ is equivalent to 
$(\pi ,\mathcal{O}(D_{(m,0,0)}))$ (see (\ref{eq:pi}) for definition) 
via the $(\mathbb{R}^m\times {\bf H}_{n-m})$-intertwining operator 
\begin{align*}
\Psi ^{\vee}:\mathcal{O}(D_{(m,0,0)})\stackrel{\sim}{\to }
	\mathcal{O}(\mathbb{C}^m\times {\bf H}_{n-m}^{\mathbb{C}}),\quad 
f\mapsto \Psi ^{\vee}(f):=f\circ \Psi . 
\end{align*}

We set $M_{(m,0,0)}:=G/H_{(m,0,0)}$ 
and consider the restriction of the quasi-regular representation $\pi _{(m,0,0)}$ of $G$ 
to $Q_{(m,0,0)}\simeq \mathbb{R}^m\times {\bf H}_{n-m}$. 
This is equivalent to the regular representation $\rho$ of $\mathbb{R}^m\times {\bf H}_{n-m}$ 
on $L^2(\mathbb{R}^m\times {\bf H}_{n-m})$, defined by the same as (\ref{eq:representation}), 
via the intertwining operator 
\begin{align*}
\Psi ^{\vee}_{\mathbb{R}}:L^2(M_{(m,0,0)})\stackrel{\sim }{\to }L^2(\mathbb{R}^m\times {\bf H}_{n-m}),\quad 
f\mapsto f\circ \Psi |_{\mathbb{R}^m\times {\bf H}_{n-m}}. 
\end{align*}

Our concern here is to construct a continuous and injective operator from 
$L^2(\mathbb{R}^m\times {\bf H}_{n-m})$ to $\mathcal{O}(\mathbb{C}^m\times {\bf H}_{n-m}^{\mathbb{C}})$. 
The key ingredient of our construction is 
to use the heat kernel on $\mathbb{R}^m$ and that on ${\bf H}_{n-m}$. 
Then, let us give a short summary on them, which is based on the literature \cite{krotz,thangavelu}. 

The heat kernel on $\mathbb{R}^m$ is the fundamental solution of the heat equation 
$\partial _tu(\boldsymbol{v},t)=\Delta _{\mathbb{R}^m}u(\boldsymbol{v},t)$ 
on $\mathbb{R}^m\times \mathbb{R}_+$ 
where $\Delta _{\mathbb{R}^m}$ designates the Laplace operator on $\mathbb{R}^m$ 
and $\mathbb{R}_+:=\{ t\in \mathbb{R}:t>0\} $, which is given for $\boldsymbol{v}\in \mathbb{R}^m$ by 
\begin{align*}
h_{\mathbb{R}^m}(\boldsymbol{v})\equiv (h_{\mathbb{R}^m})_t(\boldsymbol{v})
=(4\pi t)^{-\frac{m}{2}}e^{-\frac{1}{4t}(\boldsymbol{v}|\boldsymbol{v})}. 
\end{align*}
Since $h_{\mathbb{R}^m}:\mathbb{R}^m\to \mathbb{R}_+$ is analytic, 
this can be extended holomorphically to $\mathbb{C}^m$, still denoted by $h_{\mathbb{R}^m}$. 
For $f\in L^2(\mathbb{R}^m)$, 
we define a function $B_{\mathbb{R}^m}f\equiv (B_{\mathbb{R}^m})_tf$ on $\mathbb{C}^m$ by 
\begin{align}
(B_{\mathbb{R}^m}f)(\boldsymbol{v})
:=\int _{\mathbb{R}^m}f(\boldsymbol{s})h_{\mathbb{R}^m}(\boldsymbol{v}-\boldsymbol{s})\,d\boldsymbol{s}
\quad (\boldsymbol{v}\in \mathbb{C}^m) 
\label{eq:segal-bargmann}
\end{align}
where $d\boldsymbol{s}$ is the Lebesgue measure on $\mathbb{R}^m$. 
Indeed, the right-hand side of (\ref{eq:segal-bargmann}) is absolutely convergent 
for any $\boldsymbol{v}\in \mathbb{C}^m$, 
and $B_{\mathbb{R}^m}$ defines a map from $L^2(\mathbb{R}^m)$ to $\mathcal{O}(\mathbb{C}^m)$. 
It is well-known: 

\begin{fact}
\label{fact:segal-bargmann}
The map 
$B_{\mathbb{R}^m}:L^2(\mathbb{R}^m)\to \mathcal{O}(\mathbb{C}^m)$ 
is a continuous and injective $\mathbb{R}^m$-intertwining operator. 
\end{fact}

The operator $B_{\mathbb{R}^m}$ is called the hert kernel transform on $\mathbb{R}^m$ 
and is also known as the Segal--Bargmann transform. 

The heat kernel transform has been studied for the Heisenberg group. 
Let $h_{{\bf H}_{n-m}}\equiv (h_{{\bf H}_{n-m}})_t$ be the heat kernel on ${\bf H}_{n-m}$, 
which is the fundamental solution of the heat equation $\partial _tu(g,t)=\Delta _{{\bf H}_{n-m}}u(g,t)$ 
on ${\bf H}_{n-m}\times \mathbb{R}_+$ 
where $\Delta _{{\bf H}_{n-m}}$ designates the Laplace operator on ${\bf H}_{n-m}$. 
An explicit description of $h_{{\bf H}_{n-m}}$ is referred to \cite[Theorem 2.8.1]{thangavelu} 
or \cite[(2.2.1)]{krotz}. 
Similarly to the case of $\mathbb{R}^m$, 
$h_{{\bf H}_{n-m}}$ is an analytic function on ${\bf H}_{n-m}$, 
and then can be extended holomorphically to the complexification ${\bf H}_{n-m}^{\mathbb{C}}$. 
For $f\in L^2({\bf H}_{n-m})$, 
we define a function $B_{{\bf H}_{n-m}}f\equiv (B_{{\bf H}_{n-m}})_tf$ on ${\bf H}_{n-m}^{\mathbb{C}}$ by 
\begin{align}
(B_{{\bf H}_{n-m}}f)(g)
:=\int _{{\bf H}_{n-m}}f(k)h_{{\bf H}_{n-m}}(k^{-1}g)\,dk\quad (g\in {\bf H}_{n-m}^{\mathbb{C}}) 
\label{eq:heat kernel transform}
\end{align}
where $dk$ is a Haar measure on ${\bf H}_{n-m}$. 
Then, $B_{{\bf H}_{n-m}}$ defines a map from $L^2({\bf H}_{n-m})$ 
to $\mathcal{O}({\bf H}_{n-m}^{\mathbb{C}})$ and is called the heat kernel transform on ${\bf H}_{n-m}$. 
We know: 

\begin{fact}[cf. {\cite[Section 3.1]{krotz}}]
\label{fact:heat kernel transform}
The map 
$B_{{\bf H}_{n-m}}:L^2({\bf H}_{n-m})\to \mathcal{O}({\bf H}_{n-m}^{\mathbb{C}})$ 
is a continuous and injective ${\bf H}_{n-m}$-intertwining operator. 
\end{fact}

We are ready to construct our operator from $L^2(\mathbb{R}^m\times {\bf H}_{n-m})$ 
to $\mathcal{O}(\mathbb{C}^m\times {\bf H}_{n-m}^{\mathbb{C}})$. 
In the following, we always fix a positive number $t$. 
For $f\in L^2(\mathbb{R}^m\times {\bf H}_{n-m})$, 
we define a function $Bf$ on $\mathbb{C}^m\times {\bf H}_{n-m}^{\mathbb{C}}$ by 
\begin{align}
\label{eq:B}
(Bf)(\boldsymbol{v},x)
:=\int _{\mathbb{R}^m\times {\bf H}_{n-m}}f(\boldsymbol{s},k)
	h_{\mathbb{R}^m}(\boldsymbol{v}-\boldsymbol{s})h_{{\bf H}_{n-m}}(k^{-1}x)\,d\boldsymbol{s}dk 
\end{align}
for $(\boldsymbol{v},x)\in \mathbb{C}^m\times {\bf H}_{n-m}^{\mathbb{C}}$. 
We verify: 

\begin{lemma}
\label{lem:absolutely convergent}
The right-hand side of (\ref{eq:B}) is absolutely convergent for any 
$(\boldsymbol{v},x)\in \mathbb{C}^m\times {\bf H}_{n-m}^{\mathbb{C}}$. 
\end{lemma}

\begin{proof}
As 
$h_{\mathbb{R}^m}(\boldsymbol{v})=(4\pi t)^{-m/2}e^{-(\boldsymbol{v}|\boldsymbol{v})/4t}$ 
($\boldsymbol{v}\in \mathbb{C}^m$), 
a direct computation shows that 
\begin{align}
\int _{\mathbb{R}^m}|h_{\mathbb{R}^m}(\boldsymbol{v}-\boldsymbol{s})|^2\,d\boldsymbol{s}
=\int _{\mathbb{R}^m}|h_{\mathbb{R}^m}(\boldsymbol{s})|^2\,d\boldsymbol{s}<\infty 
\label{eq:heat kernel R}
\end{align}
for any $\boldsymbol{v}\in \mathbb{C}^m$. 
On the other hand, 
\cite[(3.1.1)]{krotz} explains 
\begin{align}
\int _{{\bf H}_{n-m}}|h_{{\bf H}_{n-m}}(k^{-1}x)|^2\,dk<\infty 
\label{eq:heat kernel H}
\end{align}
for any $x\in {\bf H}_{n-m}^{\mathbb{C}}$. 
Then, we obtain 
\begin{multline}
c(\boldsymbol{v},x)
:=\left( \int _{\mathbb{R}^m\times {\bf H}_{n-m}}
	|h_{\mathbb{R}^m}(\boldsymbol{v}-\boldsymbol{s})h_{{\bf H}_{n-m}}(k^{-1}x)|^2\,d\boldsymbol{s}dk 
\right) ^{\frac{1}{2}}\\
=\left( 
	\int _{\mathbb{R}^m}|h_{\mathbb{R}^m}(\boldsymbol{v}-\boldsymbol{s})|^2\,d\boldsymbol{s}
\right) ^{\frac{1}{2}}
\left( 
	\int _{{\bf H}_{n-m}}|h_{{\bf H}_{n-m}}(k^{-1}x)|^2\,dk
\right) ^{\frac{1}{2}}<\infty 
\label{eq:c}
\end{multline}
for any $(\boldsymbol{v},x)\in \mathbb{C}^m\times {\bf H}_{n-m}^{\mathbb{C}}$. 
Using H\"older's inequality, we get 
\begin{align}
\int _{\mathbb{R}^m\times {\bf H}_{n-m}}|f(\boldsymbol{s},k)
	h_{\mathbb{R}^m}(\boldsymbol{v}-\boldsymbol{s})h_{{\bf H}_{n-m}}(k^{-1}x)|\,d\boldsymbol{s}dk 
\leq c(\boldsymbol{v},x)\| f\| <\infty 
\label{eq:convergent}
\end{align}
for any $f\in L^2(\mathbb{R}^m\times {\bf H}_{n-m})$ 
and any $(\boldsymbol{v},x)\in \mathbb{C}^m\times {\bf H}_{n-m}^{\mathbb{C}}$. 
Here, $\| \cdot \|$ stands for the norm on $L^2(\mathbb{R}^m\times {\bf H}_{n-m})$. 
\end{proof}

By definition (\ref{eq:B}), 
$Bf$ is a holomorphic function on $\mathbb{C}^m\times {\bf H}_{n-m}^{\mathbb{C}}$, 
and then we get a map $B$ from $L^2(\mathbb{R}^m\times {\bf H}_{n-m})$ 
to $\mathcal{O}(\mathbb{C}^m \times {\bf H}_{n-m}^{\mathbb{C}})$. 

In what follows, we prove: 

\begin{theorem}
\label{thm:B}
The map 
$B:L^2(\mathbb{R}^m\times {\bf H}_{n-m})\to \mathcal{O}(\mathbb{C}^m \times {\bf H}_{n-m}^{\mathbb{C}})$ 
is a continuous and injective $(\mathbb{R}^m\times {\bf H}_{n-m})$-intertwining operator. 
\end{theorem}

\begin{proof}
Let us divide our proof into three steps as follows. 


\begin{step}
\label{step:continuous}
We show first that $B$ is a continuous map. 
We recall that $\mathcal{O}(\mathbb{C}^m\times {\bf H}_{n-m}^{\mathbb{C}})$ 
carries a Fr\'echet topology by the uniform convergence on compact sets 
in $\mathbb{C}^m\times {\bf H}_{n-m}^{\mathbb{C}}$. 
Namely, for a compact set $U$ in $\mathbb{C}^m\times {\bf H}_{n-m}^{\mathbb{C}}$, 
a semi-norm $|f|_U$ for $f\in \mathcal{O}(\mathbb{C}^m\times {\bf H}_{n-m}^{\mathbb{C}})$ is defined by 
$|f|_U:=\max _{(\boldsymbol{v},x)\in U}|f(\boldsymbol{v},x)|$. 

Let us take any $f,f_0\in L^2(\mathbb{R}^m\times {\bf H}_{n-m})$. 
The inequality (\ref{eq:convergent}) implies 
\begin{align*}
|(Bf-Bf_0)(\boldsymbol{v},x)|
\leq c(\boldsymbol{v},x)\| f-f_0\| 
\end{align*}
for $(\boldsymbol{v},x)\in \mathbb{C}^m\times {\bf H}_{n-m}^{\mathbb{C}}$. 
Then, we estimate 
\begin{align}
|Bf-Bf_0|_U
=\max_{(\boldsymbol{v},x)\in U}|(Bf-Bf_0)(\boldsymbol{v},x)|
\leq \left( \max_{(\boldsymbol{v},x)\in U}c(\boldsymbol{v},x)\right) \| f-f_0\| 
\label{eq:frechet}
\end{align}
for any compact set $U$ in $\mathbb{C}^m\times {\bf H}_{n-m}^{\mathbb{C}}$. 
Here, 
the property (\ref{eq:heat kernel R}) on the heat kernel $h_{\mathbb{R}^m}$ on $\mathbb{R}^m$ 
and \cite[(3.1.1)]{krotz} on the heat kernel $h_{{\bf H}_{n-m}}$ on ${\bf H}_{n-m}$ show 
\begin{align*}
&\max_{(\boldsymbol{v},x)\in U}c(\boldsymbol{v},x)\\
&\leq \left( \int _{\mathbb{R}^m}|h_{\mathbb{R}^m}(\boldsymbol{s})|^2\,d\boldsymbol{s}\right) ^{\frac{1}{2}}
	\left( \max_{(\boldsymbol{v},x)\in U}
		\int _{{\bf H}_{n-m}}|h_{{\bf H}_{n-m}}(k^{-1}x)|^2\,dk
	\right) ^{\frac{1}{2}}
<\infty . 
\end{align*}
Hence, (\ref{eq:frechet}) implies that $B$ is continuous. 
\end{step}


\begin{step}
\label{step:injective}
Let us take any $f\in L^2(\mathbb{R}^m\times {\bf H}_{n-m})$. 
For fixed $k\in {\bf H}_{n-m}$, 
we write $f_k(\boldsymbol{s}):=f(\boldsymbol{s},k)$ ($\boldsymbol{s}\in \mathbb{R}^m)$. 
Then, $f_k$ is a function on $\mathbb{R}^m$ and belongs to $L^2(\mathbb{R}^m)$. 
For $\boldsymbol{v}\in \mathbb{C}^m$, 
we define a function $\varphi _f(\boldsymbol{v})$ on ${\bf H}_{n-m}$ by 
\begin{align*}
\varphi _f(\boldsymbol{v})(k)
:=\int _{\mathbb{R}^m}f_k(\boldsymbol{s})h_{\mathbb{R}^m}(\boldsymbol{v}-\boldsymbol{s})\,d\boldsymbol{v}
\quad (k\in {\bf H}_{n-m}). 
\end{align*}
Then, $\varphi _f(\boldsymbol{v})$ lies in $L^2({\bf H}_{n-m})$ for any $\boldsymbol{v}\in \mathbb{C}^m$ 
because 
\begin{align*}
&\int _{{\bf H}_{n-m}}|\varphi _f(\boldsymbol{v})(k)|^2\,dk\\
&\leq \int _{{\bf H}_{n-m}}\left( 
	\int _{\mathbb{R}^m}|f(\boldsymbol{s},k)|^2\,d\boldsymbol{s}
\right) \left( 
	\int _{\mathbb{R}^m}|h_{\mathbb{R}^m}(\boldsymbol{v}-\boldsymbol{s})|^2\,d\boldsymbol{s}
\right) \,dk\\
&=\left( 
	\int _{\mathbb{R}^m}|h_{\mathbb{R}^m}(\boldsymbol{s})|^2\,d\boldsymbol{s}
\right) \| f\| 
<\infty . 
\end{align*}

Let us prove that $B$ is injective. 
Suppose that $f\in L^2(\mathbb{R}^m\times {\bf H}_{n-m})$ satisfies $Bf=0$, 
namely, 
$(Bf)(\boldsymbol{v},x)=0$ for any $(\boldsymbol{v},x)\in \mathbb{C}^m\times {\bf H}_{n-m}^{\mathbb{C}}$. 
In view of (\ref{eq:heat kernel transform}), 
$(Bf)(\boldsymbol{v},x)$ is expressed as 
\begin{align*}
(Bf)(\boldsymbol{v},x)
&=\int _{{\bf H}_{n-m}}\varphi _f(\boldsymbol{v})(k)h_{{\bf H}_{n-m}}(k^{-1}x)\,dk
=[B_{{\bf H}_{n-m}}\varphi _f(\boldsymbol{v})](x). 
\end{align*}
Then, the assumption gives rise to the relation $[B_{{\bf H}_{n-m}}\varphi _f(\boldsymbol{v})](x)=0$ 
for any $(\boldsymbol{v},x)\in \mathbb{C}^m\times {\bf H}_{n-m}^{\mathbb{C}}$. 
Thanks to Fact \ref{fact:heat kernel transform}, 
the function $\varphi _f(\boldsymbol{v})\in L^2({\bf H}_{n-m})$ must be zero 
for any $\boldsymbol{v}\in \mathbb{C}^m$, 
in particular, $\varphi _f(\boldsymbol{v})(k)=0$ 
for any $\boldsymbol{v}\in \mathbb{C}^m$ and any $k\in {\bf H}_{n-m}$. 
Moreover, it follows from (\ref{eq:segal-bargmann}) that 
$\varphi _f(\boldsymbol{v})(k)$ coincides with $(B_{\mathbb{R}^m}f_k)(\boldsymbol{v})$. 
Then, we get $(B_{\mathbb{R}^m}f_k)(\boldsymbol{v})=0$ 
for any $\boldsymbol{v}\in \mathbb{C}^m$ and any $k\in {\bf H}_{n-m}$. 
By Fact \ref{fact:segal-bargmann}, 
this means $f_k=0$ for any $k\in {\bf H}_{n-m}$, in particular, $f_k(\boldsymbol{s})=0$ 
for any $(\boldsymbol{s},k)\in \mathbb{R}^m\times {\bf H}_{n-m}$. 
In conclusion, we have proved $f=0$ since $f_k(\boldsymbol{s})=f(\boldsymbol{s},k)$, 
from which $B$ is injective. 
\end{step}


\begin{step}
\label{step:intertwiner}
Finally, let us verify $B(\rho (\boldsymbol{y},g)f)=\rho (\boldsymbol{y},g)(Bf)$ 
for any $(\boldsymbol{y},g)\in \mathbb{R}^m\times {\bf H}_{n-m}$ 
and any $f\in L^2(\mathbb{R}^m\times {\bf H}_{n-m})$ (see (\ref{eq:representation}) for definition). 
Indeed, we obtain 
\begin{align*}
&B(\rho (\boldsymbol{y},g)f)(\boldsymbol{v},x)\\
&=\int _{\mathbb{R}^m\times {\bf H}_{n-m}}[\rho (\boldsymbol{y},g)f](\boldsymbol{s},k)
	h_{\mathbb{R}^m}(\boldsymbol{v}-\boldsymbol{s})h_{{\bf H}_{n-m}}(k^{-1}x)\,d\boldsymbol{s}dk\\
&=\int _{\mathbb{R}^m\times {\bf H}_{n-m}}f(\boldsymbol{s}-\boldsymbol{y},g^{-1}k)
	h_{\mathbb{R}^m}(\boldsymbol{v}-\boldsymbol{s})h_{{\bf H}_{n-m}}(k^{-1}x)\,d\boldsymbol{s}dk\\
&=\int _{\mathbb{R}^m\times {\bf H}_{n-m}}f(\boldsymbol{s},k)
	h_{\mathbb{R}^m}(\boldsymbol{v}-(\boldsymbol{s}+\boldsymbol{y}))
		h_{{\bf H}_{n-m}}(((gk)^{-1}x)\,d\boldsymbol{s}dk\\
&=\int _{\mathbb{R}^m\times {\bf H}_{n-m}}f(\boldsymbol{s},k)
	h_{\mathbb{R}^m}((\boldsymbol{v}-\boldsymbol{y})-\boldsymbol{s})
		h_{{\bf H}_{n-m}}(k^{-1}(g^{-1}x))\,d\boldsymbol{s}dk\\
&=(Bf)(\boldsymbol{v}-\boldsymbol{y},g^{-1}x)\\
&=\rho (\boldsymbol{y},g)(Bf)(\boldsymbol{v},x)
	\quad ((\boldsymbol{v},x)\in \mathbb{C}^m\times {\bf H}_{n-m}^{\mathbb{C}}). 
\end{align*}
\end{step}

Consequently, Theorem \ref{thm:B} follows from Steps \ref{step:continuous}--\ref{step:intertwiner}. 
\end{proof}

\subsection{Equivalence between (\ref{item:visible-q}) and (\ref{item:condition-h})}
\label{subsec:visible-h}

In this subsection, we prove the equivalence between (\ref{item:visible-q}) and (\ref{item:condition-h}). 
By Proposition \ref{prop:classification}, 
it suffices to consider $H=H_{(0,0,1)}$, $H_{(p,q,1)}$, $H_{(m,0,1)}$ and $H_{(m,0,0)}$. 
Here, we have already proved that condition (\ref{item:visible-q}) holds 
if $H=H_{(0,0,1)}$ in Theorem \ref{thm:visible-(0,0,1)}; 
if $H=H_{(p,q,1)}$ in Theorem \ref{thm:visible-(p,q,1)}; 
if $H=H_{(m,0,1)}$ in Theorem \ref{thm:visible-(m,0,1)}; 
and if $H=H_{(n,0,0)}$ in Theorem \ref{thm:visible-(n,0,0)}. 
Then, let us consider the remaining case, namely, 
the $Q_{(m,0,0)}$-action on $D_{(m,0,0)}=G^{\mathbb{C}}/H_{(m,0,0)}^{\mathbb{C}}$ for $1\leq m<n$. 

Before that, 
let us focus on the regular representation $(\rho ,L^2(\mathbb{R}^m\times {\bf H}_{n-m}))$. 
It follows from (\ref{item:mf}) $\Leftrightarrow $ (\ref{item:condition-h}) 
(see Section \ref{subsec:mf-h}) that the quasi-regular representation 
$(\pi _{(m,0,0)},L^2(M_{(m,0,0)}))$ is not multiplicity-free as a representation of $G={\bf H}_n$. 
In particular, 
we know that the multiplicity $m_{\pi _{(m,0,0)}}(\xi (\boldsymbol{0},\boldsymbol{0},\gamma ))$ 
of $\tau _{\xi (\boldsymbol{0},\boldsymbol{0},\gamma )}\in \widehat{G}$ is infinite 
for any $\xi (\boldsymbol{0},\boldsymbol{0},\gamma )\in \mathfrak{q}^*_{(m,0,0)}$ with $\gamma \neq 0$. 
Thus, $\pi _{(m,0,0)}$ is of infinite multiplicities as a representation of $Q_{(m,0,0)}$, 
and then, of $\mathbb{R}^m\times {\bf H}_{n-m}$ 
under the identification $Q_{(m,0,0)}\simeq \mathbb{R}^m\times {\bf H}_{n-m}$. 
As $\pi _{(m,0,0)}$ is equivalent to $\rho $ 
by the $(\mathbb{R}^m\times {\bf H}_{n-m})$-intertwining operator $\Psi _{\mathbb{R}}^{\vee}$, 
we get: 

\begin{lemma}
\label{lem:infinite multiplicities}
The regular representation $(\rho ,L^2(\mathbb{R}^m\times {\bf H}_{n-m}))$ 
of $\mathbb{R}^m\times {\bf H}_{n-m}$ is of infinite multiplicities. 
\end{lemma}

Under the preparation, we prove: 

\begin{theorem}
\label{thm:(m,0,0)-action}
If $m<n$, then 
the $Q_{(m,0,0)}$-action on $D_{(m,0,0)}$ is not strongly visible. 
\end{theorem}

\begin{proof}
Suppose the $Q_{(m,0,0)}$-action on $D_{(m,0,0)}$ is strongly visible. 
By Fact \ref{fact:propagation} (see also Theorem \ref{thm:multiplicity-free}), 
the representation $(\pi ,\mathcal{O}(D_{(m,0,0)}))$ of $Q_{(m,0,0)}$ is multiplicity-free. 
This implies that 
$(\rho ,\mathcal{O}(\mathbb{C}^m\times {\bf H}_{n-m}^{\mathbb{C}}))$ turns out to be multiplicity-free 
as a representation of $\mathbb{R}^m\times {\bf H}_{n-m}\simeq Q_{(m,0,0)}$ 
via the intertwining operator $\Psi ^{\vee}$. 
This means that any unitary representation of $\mathbb{R}^m\times {\bf H}_{n-m}$ 
realized in $(\rho ,\mathcal{O}(\mathbb{C}^m\times {\bf H}_{n-m}^{\mathbb{C}}))$ 
in the sense of \cite[Definition 2.1]{ko13} is multiplicity-free (see \cite[Definition 1.5.3]{ko05}). 
Here, Theorem \ref{thm:B} asserts that 
the unitary representation $(\rho ,L^2(\mathbb{R}^m\times {\bf H}_{n-m}))$ is realized 
in $(\rho ,\mathcal{O}(\mathbb{C}^m\times {\bf H}_{n-m}^{\mathbb{C}}))$ 
via the operator $B$ (see (\ref{eq:B}) for definition). 
Then, $(\rho ,L^2(\mathbb{R}^m\times {\bf H}_{n-m}))$ must be multiplicity-free 
as a representation of $\mathbb{R}^m\times {\bf H}_{n-m}$, 
which contradicts to Lemma \ref{lem:infinite multiplicities}. 

Therefore, the $Q_{(m,0,0)}$-action on $D_{(m,0,0)}$ is not strongly visible. 
\end{proof}

\begin{figure}[htbp]
\begin{align*}
\begin{array}{ccccccc}
L^2(M_{(m,0,0)}) & \dashrightarrow & \mathcal{O}(D_{(m,0,0)}) \\[1mm]
\Psi _{\mathbb{R}}^{\vee} ~\downarrow & & \downarrow ~\Psi ^{\vee} \\[1mm]
L^2(\mathbb{R}^m\times {\bf H}_{n-m}) & \xrightarrow{B} & \mathcal{O}(\mathbb{C}^m\times {\bf H}_{n-m}^{\mathbb{C}})
\end{array}
\end{align*}
\caption{$(\mathbb{R}^m\times {\bf H}_{n-m})$-intertwining operator}
\label{fig:intertwining operator}
\end{figure}

\begin{proof}[Proof of (\ref{item:visible-q}) $\Leftrightarrow $ (\ref{item:condition-h}) 
in Theorem \ref{thm:criterion}]
As mentioned above, 
we have already shown the implication (\ref{item:condition-h}) $\Rightarrow $ (\ref{item:visible-q}). 
The opposite is also true by Theorem \ref{thm:(m,0,0)-action}. 
Consequently, 
the proof of Theorem \ref{thm:criterion} has been accomplished. 
\end{proof}

\subsection{Slice for $Q$-action on $D$ and the support of $\pi _H$}
\label{subsec:support}

We end this paper 
by considering the multiplicity-free irreducible decomposition (\ref{eq:irreducible decomposition}) 
of the quasi-regular representation $\pi _H$ of $G$. 
In accordance with Lemma \ref{lem:mf-equiv} and 
the equivalence (\ref{item:mf}) $\Leftrightarrow $ (\ref{item:condition-h}) in Theorem \ref{thm:criterion}, 
we may always consider non-trivial connected closed subgroups $H$ 
as $H_{(0,0,1)}$, $H_{(p,q,1)}$, $H_{(m,0,1)}$ and $H_{(n,0,0)}$. 

As we have seen in Fact \ref{fact:corwin-greenleaf}, 
the irreducible decomposition of $\pi _H$ can be described in terms of  the coadjoint orbits. 
More precisely, 
the irreducible representations of $G$ occurring in (\ref{eq:irreducible decomposition}) 
are parameterized by the $G$-orbit space $(G\cdot \mathfrak{q}^*)/G$ 
of the subset $G\cdot \mathfrak{q}^*$ of $\mathfrak{g}^*$. 
Then, we focus on $(G\cdot \mathfrak{q}^*)/G$ below. 

For a cross-section $R$ of the coadjoint orbit space $\mathfrak{g}^*/G$, 
the set $R\cap (G\cdot \mathfrak{q}^*)$ becomes a cross-section of $(G\cdot \mathfrak{q}^*)/G$. 
In Lemma \ref{lem:coadjoint orbit}, we have taken $R$ as 
\begin{align}
\tag{\ref{eq:r}}
R&=\{ \xi (\boldsymbol{0},\boldsymbol{0},\gamma ):\gamma \in \mathbb{R}^{\times }\} 
\sqcup \{ \xi (\boldsymbol{\alpha },\boldsymbol{\beta },0)
:\boldsymbol{\alpha },\boldsymbol{\beta }\in \mathbb{R}^n\} . 
\end{align}
Since the multiplicity $m_{\pi _H}(\xi )$ of $\tau _{\xi }\in \widehat{G}$ corresponding to 
$\xi \in (G\cdot \mathfrak{q}^*)/G $ does not depend on the choice of the cross-section 
(see Fact \ref{fact:corwin-greenleaf}), 
we may fix $R$ as (\ref{eq:r}) in our argument below. 
Now, we choose 
\begin{align}
\label{eq:r-q}
R(\mathfrak{q}^*):=R\cap (G\cdot \mathfrak{q}^*) 
\end{align}
as a cross-section  of $(G\cdot \mathfrak{q}^*)/G$. 

\begin{lemma}
\label{lem:r-q}
The set $R(\mathfrak{q}^*)$ coincides with $R\cap \mathfrak{q}^*$. 
\end{lemma}

\begin{proof}
Let $H\equiv H_{(k,\ell ,\varepsilon )}$ be $H_{(0,0,1)},H_{(p,q,1)}$ or $H_{(m,0,1)}$. 
We have seen in (\ref{eq:orbit space-normal}) that 
$G\cdot \mathfrak{q}_{(k,\ell ,\varepsilon )}^*=\mathfrak{q}_{(k,\ell ,\varepsilon )}^*$, 
from which we obtain 
$R(\mathfrak{q}_{(k,\ell ,\varepsilon )}^*)=R\cap \mathfrak{q}_{(k,\ell ,\varepsilon )}^*$. 
Since $\mathfrak{q}_{(k,\ell ,\varepsilon )}^*$ is a subspace of 
$\operatorname{span}_{\mathbb{R}}\{ X_i^*,Y_j^*:1\leq i,j\leq n\} $, 
$R(\mathfrak{q}_{(k,\ell ,\varepsilon )}^*)$ is given by 
\begin{align}
\label{eq:r-q-normal}
R(\mathfrak{q}_{(k,\ell ,\varepsilon )}^*)
=\{ \xi (\boldsymbol{\alpha },\boldsymbol{\beta },0)
:\boldsymbol{\alpha },\boldsymbol{\beta }\in \mathbb{R}^n\} \cap \mathfrak{q}_{(k,\ell ,\varepsilon )}^*
=\mathfrak{q}_{(k,\ell ,\varepsilon )}^*. 
\end{align}

Let $H$ be $H_{(n,0,0)}$. 
Then, $\mathfrak{q}_{(n,0,0)}^*=\{ \xi (\boldsymbol{0},\boldsymbol{\beta },\gamma ):
\boldsymbol{\beta }\in \mathbb{R}^n,\gamma \in \mathbb{R}\} $. 
In view of Lemma \ref{lem:coadjoint orbit}, 
$G\cdot \mathfrak{q}_{(n,0,0)}^*$ is expressed as 
\begin{align*}
G\cdot \mathfrak{q}_{(n,0,0)}^*
&=\{ \xi (\boldsymbol{x},\boldsymbol{y},\gamma )
	:\boldsymbol{x},\boldsymbol{y}\in \mathbb{R}^n,\gamma \in \mathbb{R}^{\times }\}
		\sqcup \{ \xi (\boldsymbol{0},\boldsymbol{\beta },0):\boldsymbol{\beta }\in \mathbb{R}^n\} , 
\end{align*}
from which the orbit space is written as 
\begin{align*}
(G\cdot \mathfrak{q}_{(n,0,0)}^*)/G
&=\{ G\cdot \xi (\boldsymbol{0},\boldsymbol{0},\gamma ):\gamma \in \mathbb{R}^{\times }\} 
	\sqcup 
	\{ G\cdot \xi (\boldsymbol{0},\boldsymbol{y},0):\boldsymbol{y}\in \mathbb{R}^n\} . 
\end{align*}
Hence, $R(\mathfrak{q}_{(n,0,0)}^*)$ is of the form 
\begin{align}
\label{eq:r-q-(n,0,0)}
R(\mathfrak{q}_{(n,0,0)}^*)
=\{ \xi (\boldsymbol{0},\boldsymbol{0},\gamma ):\gamma \in \mathbb{R}^{\times}\} 
	\sqcup \{ \xi (\boldsymbol{0},\boldsymbol{\beta },0):\boldsymbol{\beta }\in \mathbb{R}^n\} , 
\end{align}
which coincides with $R\cap \mathfrak{q}_{(n,0,0)}^*$. 
\end{proof}

Lemma \ref{lem:r-q} explains that $R(\mathfrak{q}^*)$ is a subset of the real vector space $\mathfrak{q}^*$. 
Then, there exists a minimal finite subset 
$\mathcal{B}(R(\mathfrak{q}^*))=\{ \xi _1,\ldots ,\xi _r\} $ in $\mathfrak{q}^*$ 
such that 
$\mathcal{B}(R(\mathfrak{q}^*))$ is linearly independent in $\mathfrak{q}^*$ and 
arbitrary element of $R(\mathfrak{q}^*)$ is written 
as the linear combination of $\mathcal{B}(R(\mathfrak{q}^*))$. 

We say that 
\begin{align*}
\operatorname{Supp}(\pi _H)
:=\{ \tau _{\xi }\in \widehat{G}:\xi \in R(\mathfrak{q}^*),m_{\pi _H}(\xi )\neq 0\} 
\end{align*}
is the {\it support} of the irreducible representations of $G$ 
occurring in (\ref{eq:irreducible decomposition}), and that the number 
\begin{align}
\label{eq:rank}
\operatorname{rank}(\operatorname{Supp}(\pi _H)):=|\mathcal{B}(R(\mathfrak{q}^*))|=r 
\end{align}
is the {\it rank} of $\operatorname{Supp}(\pi _H)$. 

\begin{theorem}
\label{thm:support}
Let $G$ be the Heisenberg group 
and $H$ a non-trivial connected closed subgroup of $G$. 
Suppose that the quasi-regular representation $\pi _H$ of $G$ is multiplicity-free. 
Then we have: 
\begin{enumerate}
	\item $\operatorname{rank}(\operatorname{Supp}(\pi _H))=\dim G/H$. 
	\item One can find a slice $S$ of dimension $\operatorname{rank}(\operatorname{Supp}(\pi _H))$ 
	for the strongly visible $Q$-action on $G^{\mathbb{C}}/H^{\mathbb{C}}$.
\end{enumerate}
\end{theorem}

\begin{proof}
Let $H\equiv H_{(k,\ell ,\varepsilon )}$ be one of $H_{(0,0,1)}$, $H_{(p,q,1)}$ or $H_{(m,0,1)}$. 
Then, we have $R(\mathfrak{q}_{(k,\ell ,\varepsilon )}^*)=\mathfrak{q}_{(k,\ell ,\varepsilon )}^*$ 
(see (\ref{eq:r-q-normal})). 
This means that $\mathcal{B}(R(\mathfrak{q}_{(k,\ell ,\varepsilon )}^*))$ is nothing but 
a basis of $\mathfrak{q}^*_{(k,\ell ,\varepsilon )}$. 
Thus, we obtain $\operatorname{rank}(\operatorname{Supp}(\pi _{H_{(k,\ell ,\varepsilon )}}))
=\dim \mathfrak{q}_{(k,\ell ,\varepsilon )}=\dim G/H_{(k,\ell ,\varepsilon )}$. 
On the other hand, 
Table \ref{table:choice} asserts that our choice of slice $S_{(k,\ell ,\varepsilon )}$ 
for the strongly visible $Q_{(k,\ell ,\varepsilon )}$-action 
on $G^{\mathbb{C}}/H_{(k,\ell ,\varepsilon )}^{\mathbb{C}}$ satisfies 
$\dim S_{(k,\ell ,\varepsilon )}=\dim \sqrt{-1}\mathfrak{q}_{(k,\ell ,\varepsilon )}
=\dim G/H_{(k,\ell ,\varepsilon )}$. 
Hence, we get $\operatorname{rank}(\operatorname{Supp}(\pi _{H_{(k,\ell ,\varepsilon )}}))
=\dim G/H_{(k,\ell ,\varepsilon )}
=\dim S_{(k,\ell ,\varepsilon )}$. 

Next, let $H$ be $H_{(n,0,0)}$. 
Since $R(\mathfrak{q}_{(n,0,0)}^*)$ is given by (\ref{eq:r-q-(n,0,0)}), 
we can take $\mathcal{B}(R(\mathfrak{q}_{(n,0,0)}^*))$ as $\{ Y_1^*,\ldots ,Y_n^*,Z^*\} $. 
Clearly, this is also a basis of $\mathfrak{q}_{(n,0,0)}^*$, 
from which 
$\operatorname{rank}(\operatorname{Supp}(\pi _{H_{(n,0,0)}}))
=\dim \mathfrak{q}_{(n,0,0)}^*=\dim G/H_{(n,0,0)}$. 
Further, it follows from Theorem \ref{thm:visible-(n,0,0)} that 
the $Q_{(n,0,0)}$-action on $G^{\mathbb{C}}/H_{(n,0,0)}^{\mathbb{C}}$ is strongly visible 
with slice $S_{(n,0,0)}\simeq \exp \sqrt{-1}\mathfrak{q}_{(n,0,0)}$. 
Thus, we have $\dim S_{(n,0,0)}=\dim \sqrt{-1}\mathfrak{q}_{(n,0,0)}=\dim G/H_{(n,0,0)}$. 
Hence, we find out that $\operatorname{rank}(\operatorname{Supp}(\pi _{H_{(n,0,0)}}))
=\dim G/H_{(n,0,0)}=\dim S_{(n,0,0)}$, which achieves the  proof of the  theorem.
\end{proof}



\end{document}